\documentclass[11pt,oneside]{amsart}
\usepackage{fullpage}

\usepackage{amsthm,amsfonts,amssymb,amsmath,amsxtra,mathtools}
\usepackage{graphicx}
\usepackage{stmaryrd}
\usepackage{enumitem}
\usepackage{tikz-cd}
\usepackage[colorlinks=true]{hyperref}
\usepackage{mathrsfs}

\DeclareFontFamily{U}{matha}{\hyphenchar\font45}
\DeclareFontShape{U}{matha}{m}{n}{
      <5> <6> <7> <8> <9> <10> gen * matha
      <10.95> matha10 <12> <14.4> <17.28> <20.74> <24.88> matha12
      }{}
\DeclareSymbolFont{matha}{U}{matha}{m}{n}
\DeclareFontFamily{U}{mathx}{\hyphenchar\font45}
\DeclareFontShape{U}{mathx}{m}{n}{
      <5> <6> <7> <8> <9> <10>
      <10.95> <12> <14.4> <17.28> <20.74> <24.88>
      mathx10
      }{}
\DeclareSymbolFont{mathx}{U}{mathx}{m}{n}

\DeclareMathSymbol{\obot}         {2}{matha}{"6B}

\newtheorem{theorem}{Theorem}[section]
\newtheorem{corollary}{Corollary}[theorem]
\newtheorem{proposition}[theorem]{Proposition}
\newtheorem{lemma}[theorem]{Lemma}
\newtheorem{remark}[theorem]{Remark}

\theoremstyle{definition}
\newtheorem{definition}[theorem]{Definition}

\numberwithin{equation}{section}

\newcommand{\bX}{\mathbb{X}}
\newcommand{\bx}{{\bf x}}
\newcommand{\bk}{{\bf k}}
\newcommand{\bV}{\mathbb{V}}
\newcommand{\bY}{\mathbb{Y}}
\newcommand{\by}{\mathbf{y}}
\newcommand{\bA}{\mathbb{A}}

\newcommand{\bL}{\mathbf{L}}
\newcommand{\R}{\mathbb{R}}

\newcommand{\Q}{\mathbb{Q}}
\newcommand{\Z}{\mathbb{Z}}

\newcommand{\C}{\mathbb{C}}

\renewcommand{\H}{\mathbb{H}}
\newcommand{\F}{\mathbb{F}}

\newcommand{\bG}{\mathbb{G}}

\newcommand{\bF}{\mathbf{F}}
\newcommand{\bfV}{\mathbf{V}}
\newcommand{\Oo}{\mathcal{O}}
\newcommand{\cV}{\mathcal{V}}
\newcommand{\cZ}{\mathcal{Z}}
\newcommand{\cN}{\mathcal{N}}
\newcommand{\cW}{\mathcal{W}}

\newcommand{\cM}{\mathcal{M}}
\newcommand{\cY}{\mathcal{Y}}
\newcommand{\cP}{\mathcal{P}}

\newcommand{\Spec}{\mathrm{Spec}\, }
\newcommand{\Spf}{\mathrm{Spf}\, }

\newcommand{\Herm}{\mathrm{Herm}}
\newcommand{\Hom}{\mathrm{Hom}}
\newcommand{\End}{\mathrm{End}}
\newcommand{\rU}{\mathrm{U}}
\newcommand{\Nilp}{\mathrm{Nilp}\, }
\newcommand{\Lie}{\mathrm{Lie}\, }
\begin{document}

\title[Special cycles on unitary Shimura Varieties at ramified primes]{Special cycles on the basic locus of unitary Shimura Varieties at ramified primes}

\author{Yousheng Shi}
\address{Department of Mathematics, University of Wisconsin}


%    General info
%\subjclass[2000]{Primary 54C40, 14E20; Secondary 46E25, 20C20}

\date{\today}

%\keywords{Differential geometry, algebraic geometry}

\begin{abstract}
In this paper, we study special cycles on the basic locus of certain unitary Shimura varieties over the ramified primes and their local analogues on the corresponding Rapoport-Zink spaces. We study the support and compute the dimension of these cycles.
\end{abstract}

\maketitle

\setcounter{tocdepth}{1}
\tableofcontents

\section{Introduction}
In this paper, we study the basic locus of certain unitary Shimura varieties over ramified  primes and special cycles on it. To approach this global problem we first study local special cycles on the corresponding Rapoport-Zink spaces and then apply a uniformization theorem to convert our local results on the  Rapoport-Zink spaces to global ones. Our results will have applications to Kudla's program, in particular Kudla-Rapoport type of conjectures over these ramified primes (c.f. \cite{KR1}, \cite{KR2}, \cite{LZ}, \cite{liliu2021},\cite{HSY} and \cite{Shi2}).   

We specialize to an integral model of Shimura varieties associated to $\rU(1,n-1)$ which parametrize abelian schemes with certain CM action and a compatible principal polarization. This integral model and the corresponding model of Rapoport-Zink space is first proposed by Pappas \cite{P} (see also \cite{RSZshimura}). It is flat over the base, normal and Cohen-Macaulay and has isolated singularities. One can blow up these singularities to get a model which has semi-stable reduction and has a simple moduli interpretation, see \cite{Kr}. We focus on the Pappas model in this paper but all results can be easily adjusted to the Kr\"amer model case as these models are the same outside the singularities. 

In the Rapoport-Zink spaces setting, we study the reduced locus of special cycles and compute their dimensions. As an intermediate step, we prove an isomorphism between two Rapoport-Zink spaces of different nature. In the  Shimura variety setting, we write down the uniformization theorem of the basic locus over the ramified primes and then translate our local results to global ones. We now explain our results in more detail.

\subsection{Local results}\label{subsec:introlocal}
Let $p>2$ be a prime and fix a tower of finite extensions $\Q_p\subseteq H\subseteq F_0\subset F$ where $F/F_0$ is quadratic and ramified. For any $p$-adic field $R$, we denote by $\Oo_R$ its ring of integers.
Let $\breve F$ be the completion of the maximal unramified extension of $F$. Let $\Nilp \Oo_{\breve F}$  be the categories of $\Oo_{\breve F}$-schemes $S$ on which $p$ is locally nilpotent.
For $S\in \Nilp \Oo_{\breve F}$, let $\bar{S}=S\times_{\Spec \Oo_{\breve F}} \Spec \bar{\F}_p$.
The Rapoport-Zink space $\cN^{F/H}_{(r,s)}$ is the moduli space over $\Spf \Oo_{\breve F}$ whose $S$ points are objects $(X,\iota,\lambda,\rho)$ where $X$ is a supersingular formal $p$-divisible group over $S$, $\iota:\Oo_F\rightarrow \End(X)$ is an $\Oo_F$-action on $X$ whose restriction to $\Oo_H$ is strict, $\lambda:X\rightarrow X^\vee$ is a principal polarization, and $\rho:X\times_{S} \bar{S}\rightarrow \bX\times_{\Spec \bar{\F}_p} \bar{S}$ is a map to a framing object $(\bX,\iota_\bX,\lambda_\bX) $ over $\Spec \bar{\F}_p$. We require that the Rosati involution of $\lambda$ induces on $\Oo_F$ the Galois conjugation over $\Oo_{F_0}$ and the action $\iota$ satisfies the $(r,s)$ signature condition (Definition \ref{def:signatureconditions}). See Definition \ref{def:EHmodulewithsignature} and \ref{def:RZspace} for the detailed definition of $\cN^{F/H}_{(r,s)}$. We first show the following theorem (Theorem \ref{thm:comparisonRZspaces}). 
\begin{theorem}\label{thm:introA}
Suppose that $F_0/H$ is unramified. Then there is an isomorphism
\[\mathfrak{C}:\cN^{F/H}_{(r,s)}\cong \cN^{F/F_0}_{(r,s)}.\]
\end{theorem}
The significance of the above theorem is that  $\cN^{F/\Q_p}_{(r,s)}$ can be related to unitary Shimura varieties by the uniformization theorem while $\cN^{F/F_0}_{(r,s)}$ is easier to study. From now on we mainly focus on the signature $(1,n-1)$. By \cite{RTW} we know that $\cN^{F/F_0}_{(1,n-1)}$ is representable by a formal scheme over $\Spf \Oo_{\breve F}$. Moreover there is a stratification of its reduced locus given by
\[(\cN^{F/F_0}_{(1,n-1)})_{\mathrm{red}}=\biguplus_{\Lambda} \cN_{\Lambda}^o\]
where $\Lambda$ runs over the so-called vertex lattices, see Theorem \ref{thm:incidencerelation}.

We can define special cycles on both $\cN^{F/H}_{(1,n-1)}$ and $\cN^{F/F_0}_{(1,n-1)}$. The isomorphism in Theorem \ref{thm:introA} maps special cycles in the first space to special cycles in the second. Without loss of generality we focus on $\cN^{F/F_0}_{(1,n-1)}$. Let $(\bY,\iota_\bY,\lambda_\bY)$ (resp. $(\bX,\iota_\bX,\lambda_\bX)$) be the framing object of $\cN^{F/F_0}_{(0,1)}$ (resp. $\cN^{F/F_0}_{(1,n-1)}$). Define an $F$ vector space
\[\bV:=\Hom_{\Oo_F}(\bY,\bX)\otimes_{\Z} \Q\]
of rank $n$ with the Hermitian form $h(\cdot,\cdot)$ such that for any $x,y\in \bV$ we have
\[ h(x,y)=\lambda_\bY^{-1} \circ y^\vee \circ\lambda_\bX \circ x\in \mathrm{End}_{\Oo_F}(\bY)\otimes \Q \xrightarrow{\cong} F,\]
where $y^\vee$ is the dual of $y$.
For an $\Oo_F$-lattice $\bL\subset \bV$, the associated special cycle $\cZ(\bL)$ is the subfunctor of $\cN^{F/F_0}_{(0,1)} \times_{\Spf \Oo_{\breve F} } \cN^{F/F_0}_{(1,n-1)}$ such that $\xi=(Y,\iota,\lambda_{Y}, \varrho_{Y}, {X},\iota,\lambda_{ X},\varrho_{ X})\in \cZ(\bL)(S)$ if for any $\bx\in \bL$ the quasi-homomorphism 
    \[\varrho^{-1}_{X}\circ {\bx} \circ \varrho_{ Y}: Y\times_S \bar{S}\rightarrow X\times_S \bar{S}\]
lifts to a homomorphism from $Y$ to $X$.

Any $\Oo_F$-lattice $L$ with a Hermitian form $(\cdot,\cdot)$ has a Jordan splitting 
\begin{equation}\label{eq:IntroJordandecomp}
    L=\obot_{\lambda \in \Z} L_\lambda
\end{equation}
where $\obot$ stands for orthogonal direct sum and $L_\lambda$ is $\pi^{\lambda}$-modular (see Section \ref{subsec:Hermitianlattices}). We say $L$ is integral if $(x,y)\in \Oo_F$ for any $x,y\in L$.
For an integer $t$ and a fixed Jordan decomposition as above we define 
\[L_{\geq t}=\obot_{\lambda \geq t} L_\lambda\subset L.\]
The following summarizes Theorem \ref{thmA} and \ref{thmB} and their corollaries.
\begin{theorem}\label{thm:introB}
Let $\bL\subset \bV$ be an $\Oo_F$-lattice of rank $n$.
\begin{enumerate}[leftmargin=*, label=({\roman*})]
\item $\cZ(\bL)$ is non-empty if and only if $\bL$ is integral.
\item $\cZ(\bL)_{\mathrm{red}}$ (the reduced scheme of $\cZ(\bL)$) is a union of strata $\cN_{\Lambda}^o$ where $\Lambda$ ranges over a set of vertices which can be described in terms of $\bL$.
\item Fix a Jordan decomposition of $\bL$ as in \eqref{eq:IntroJordandecomp}.
Define 
    \[\mathfrak{d}(\bL):=\left\{\begin{array}{cc}
      \mathrm{rank}_{\Oo_F}(\bL_{\geq 1})-1   & \text{if \ $\mathrm{rank}_{\Oo_F}(\bL_{\geq 1})$ is odd} \\
      \mathrm{rank}_{\Oo_F}(\bL_{\geq 1})  & \text{if\  $\mathrm{rank}_{\Oo_F}(\bL_{\geq 1})$ is even and $ \bL_{\geq 1}\otimes_\Z \Q$ is split}\\
      \mathrm{rank}_{\Oo_F}(\bL_{\geq 1})-2  & \text{if\  $\mathrm{rank}_{\Oo_F}(\bL_{\geq 1})$ is even and $ \bL_{\geq 1}\otimes_\Z \Q$ is non-split},\\
    \end{array}\right.\]
Then $\cZ(\bL)_{\mathrm{red}}$ is purely of dimension $\frac{1}{2}\mathfrak{d}(\bL)$, i.e., every irreducible component of $\cZ(\bL)_{\mathrm{red}}$ is of dimension $\frac{1}{2}\mathfrak{d}(L)$. Here we say a Hermitian space $V$ of dimension $n$ is split if
\[(-1)^{n(n-1)/2} \det(V)\in \mathrm{Nm}_{F/F_0}(F^\times).\]
Otherwise we say it is nonsplit. 
\item 
Define 
\[n_{\mathrm{odd}}=\sum_{\lambda\geq 3,\, \lambda \text{ is odd}}\mathrm{rank}_{\Oo_F}(\bL_\lambda),\]
and 
\[n_{\mathrm{even}}=\sum_{\lambda\geq 2,\, \lambda \text{ is even}}\mathrm{rank}_{\Oo_F}(\bL_\lambda).\]
Then $\cZ(\bL)_{\mathrm{red}}$ is irreducible if and only if the following two conditions hold simultaneously
\begin{enumerate}
    \item $n_{\mathrm{odd}}=0$. \ 
    \item $n_{\mathrm{even}} \leq 1$ or $n_{\mathrm{even}}=2$ and $\bL_{\geq 2}\otimes_\Z \Q$ is non-split.
\end{enumerate}
\end{enumerate}
\end{theorem}

\subsection{Global results}\label{subsec:introglobal}
In the global setting, let $F$ be a CM field with totally real subfield $F_0$ and $\Phi\subset \Hom_\Q(F,\C)$ be a CM type of $F$. Denote by $x\mapsto \bar{x}$ the Galois conjugation of $F/F_0$ and fix a $\varphi_0\in\Phi$. Define
\begin{equation}\label{eq:cVram}
    \mathscr{V}_{\mathrm{ram}}=\{\text{finite places } v \text{ of }F_0\mid v\text{ ramifies in }F\}.
\end{equation}
We assume that $\mathscr{V}_{\mathrm{ram}}$ is nonempty and every $v\in \mathscr{V}_{\mathrm{ram}}$ is unramified over $\Q$ and does not divide $2$.

Let $V$ be a $n$ dimensional $F$-vector space with a Hermitian form $(\cdot,\cdot)$ which has signature $(n-1,1)$ with respect to $\varphi_0$ and $(n,0)$ with respect to any other $\varphi\in \Phi\backslash \{\varphi_0\}$. The CM type $\Phi$ together with the signature of $V$ determines a reflex field $E$ and $F$ embeds into $E$ via $\varphi_0$. Define groups
\[Z^\Q:=\{z\in\mathrm{Res}_{F/\Q} \bG_m\mid \mathrm{Nm}_{F/F_0}(z)\in \bG_m\},\ 
G=\mathrm{Res}_{F_0/\Q}\rU(V). \]
Also define 
\[\tilde{G}:=Z^\Q\times G.\]
We can define a corresponding Hodge map $h_{\tilde G}:\C^\times\rightarrow \tilde{G}(\R)$. By choosing a  compact subgroup $K=K_{Z^\Q}\times K_G\subset \tilde{G}(\bA_f)$ where $K_{Z^\Q}$ is the maximal compact subgroup of $Z^\Q(\bA_f)$ (see \eqref{eq:K_Z}) and $K_G$ is a compact subgroup of $G(\bA_{f})$, we get a Shimura variety $S(\tilde G, h_{\tilde{G}})_K$ which has a canonical model over $\Spec E$. Further more if we assume $K_G$ is the stablizer of a self-dual lattice (see \eqref{eq:K_G}), then
\cite{RSZshimura} defined a moduli functor $\cM$ of abelian varieties with $\Oo_F$-action and a compatible principal polarization over $\Spec \Oo_E$ whose complex fiber is $S(\tilde G, h_{\tilde{G}})_K$. We review the definition in Section \ref{sec:shimuravariety}. By our assumption $K$ is of the form $K=\prod_{v} K_v$ where we take the restricted product over all finite places of $F_0$.
Throughout the paper, we use the notations
\[K_p=\prod_{v\mid p} K_v,\ K^p=\prod_{v\nmid p} K_v,\]
and similar notations with $K$ replaced by $K_G$ or $K_{Z^\Q}$.

Now assume $v_0\in \mathscr{V}_{\mathrm{ram}}$ and let $w_0$ be the place of $F$ above it. Let $p$ be the characteristic of the residue field of $F_{0,v_0}$.
Fix a finite place $\nu$ of $E$ above $v_0$ with residue field $k_\nu$.
Let $\breve{E}_\nu$ be the completion of the maximal unramified extension of $E_\nu$. We denote by $\cM_{\nu}^{ss}$ the basic locus of $\cM$ at $\nu$ and denote by $\widehat{\cM}_{\nu}^{ss}$ the completion of $\cM\times_{\Spec \Oo_{E}}\Spec \Oo_{\breve{E}_\nu}$ along $\cM_{\nu}^{ss}\times_{\Spec k_\nu}\Spec \bar{k}_\nu$.  
Then we have the following uniformization theorem which is a consequence of \cite[Theorem 6.30]{RZ} and Theorem \ref{thm:introA}.
\begin{theorem}\label{thm:introC}
Assume $v_0\in \mathscr{V}_{\mathrm{ram}}$ and $\mathscr{V}_{\mathrm{ram}}$ satisfies the condition stated after \eqref{eq:cVram}. There is an isomorphism:
\[\Theta: \tilde{G}'(\Q)\backslash \cN'\times\tilde{G}(\bA_f^p)/K^p \cong \widehat{M}_\nu^{ss}.\]
where $\tilde{G}'$ is an inner form of $\tilde{G}$ and 
\[\cN'=Z^\Q(\Q_p)/K_{Z^\Q,p}\times (\cN^{F_{w_0}/F_{0,v_0}}_{(1,n-1)} \widehat{\times}_{\Spf\Oo_{\breve{F}_{w_0}}} \Spf \Oo_{\breve{E}_\nu})\times \prod_{v\neq v_0} \rU(V)(F_{0,v})/K_{G,v}\]
where the product in the last factor is over all places of $F_0$ over $p$ not equal to $v_0$.
\end{theorem}

For a non-degenerate totally positive definite $F/F_0$-Hermitian matrix $T$, we define the special cycle $\cZ(T)$ following the definition of \cite{KR2} in Definition \ref{def:global Z T}. Now assume $T$ has rank $n$. Let $V_T$ be the Hermitian $F$-space with Gram matrix $T$ and define 
\begin{equation}\label{eq:DiffTVintro}
    \mathrm{Diff}(T,V):=\{v\text{ is a finite place of }F_0\mid V_v \text{ is not isomorphic to } (V_T)_v\}.
\end{equation}
It is well-known to experts that $\cZ(T)$ is empty when $\mathrm{Diff}(T,V)$ contains more than one element and $\cZ(T)$ is supported on $\cM^{ss}_\nu$ over finite primes $\nu$ of $E$ above $v$ if $\mathrm{Diff}(T,V)=\{v\}$. We briefly review the proof of these results (see Proposition \ref{prop:precisesupport}).
Then the following theorem is a consequence of Theorem \ref{thm:introB} and Theorem \ref{thm:introC}.
\begin{theorem}\label{thm:introD}
Assume that $T$ is a totally positive definite $F/F_0$-Hermitian matrix with values in $\Oo_F$ such that $\mathrm{Diff}(T,V)=\{v_0\}$ where $v_0\in \mathscr{V}_{\mathrm{ram}}$ and $\mathscr{V}_{\mathrm{ram}}$ satisfies the condition stated after \eqref{eq:cVram}. Then $\cZ(T)$ is supported on $\cM^{ss}_\nu$ and $\cZ(T)_{\mathrm{red}}$ is equidimensional of dimension $\frac{1}{2}\mathfrak{d}(L_{v_0})$ where $L_{v_0}$ is any Hermitian lattice over $\Oo_{F,v_0}$ whose gram matrix is $T$ and $\mathfrak{d}(L_{v_0})$ is defined as in Theorem \ref{thm:introB}. 
\end{theorem}

The paper is organized as follows. In Section \ref{sec:RZspaces}, we prove Theorem \ref{thm:introA} (Theorem \ref{thm:comparisonRZspaces}) and recall some properties of $\cN^{F/F_0}_{(1,n-1)}$ as studied in \cite{RTW}. In Section \ref{sec:localspecialcycle}, we define our local version of special cycles on Rapoport-Zink spaces and prove Theorem \ref{thm:introB} (Theorem \ref{thmA} and \ref{thmB}). In Section \ref{sec:shimuravariety}, we recall the definition of the arithmetic model of the Shimura variety studied in \cite{RSZdiagonal}. In Section \ref{sec:globalspecialcycle} we define global special cycles $\cZ(T)$ and prove Theorem \ref{thm:introC} (Theorem \ref{thm:uniformization}) and \ref{thm:introD} (Theorem \ref{thm:mainglobal}).

\noindent
{\bf Acknowledgement.} The author would like to thank Michael Rapoport for suggesting the problem and reading through early versions of the paper. We would also like to thank Patrick Daniels, Qiao He and Tonghai Yang for helpful discussions.

\section{Relative and absolute Rapoport-Zink spaces}\label{sec:RZspaces}
We use the notations as in Section \ref{subsec:introlocal}.
In this section, we define the Rapoport-Zink space $\cN^{F/H}_{(r,s)}$ and recall its basic properties from \cite{RTW} when $H=F_0$ and $(r,s)=(1,n-1)$. The space $\cN^{F/F_0}_{(1,n-1)}$ is convenient for studying special cycles. 
On the other hand $\cN^{F/\Q_p}_{(r,s)}$ shows up naturally in the uniformization theorem (see Theorem \ref{thm:uniformization}) of the basic locus of certain unitary Shimura varieties (see Section \ref{sec:shimuravariety}). We call $\cN^{F/F_0}_{(r,s)}$ (resp. $\cN^{F/\Q_p}_{(r,s)}$) a relative (resp. absolute) Rapoport-Zink space following the terminology of \cite{mihatschrelative}. 

In Theorem \ref{thm:comparisonRZspaces}, we show that for different choices of $H$, $\cN^{F/H}_{(r,s)}$ are isomorphic to each other given that $F_0/H$ is unramified. We follow the approach of \cite[Section 2.8]{liliu2021}. Alternatively one can use the method of \cite{KRZuniformization}. The analogue of Theorem \ref{thm:comparisonRZspaces} when $F/F_0$ is unramified was proved in \cite{mihatschrelative}.

\subsection{The signature condition}
Assume $F_0/H$ is unramified with degree $f$. 
% By local class field theory, it is easy to see that $F$ is Galois over $H$.
We denote the Galois conjugation of $F/F_0 $ by $x\mapsto \bar{x}$. Fix a uniformizer $\pi$ of $F$ such that $\pi_0:=\pi^2\in F_0$ and is a uniformizer of $F_0$. Let $\bk$ be the residue field of $\Oo_{F_0}$ (hence also that of $\Oo_F$) with an algebraic closure $\bar\bk$.
Let $\breve{H}$ be the completion of a maximal unramified extension of $H$ (hence also that of $F_0$) in $\breve{F}$. Let $x\mapsto \sigma(x)$ denote the Frobenius of $\breve{H}/H$. Define
\[\Psi:=\Hom_H(F_0,\breve{H}).\]
Fix a distinguished element $\psi_0\in \Psi_0$. 
Define $\psi_i=\sigma^i\circ\psi_0$ for $i \in \Z/(f\Z)$.
Then
\[\Psi=\{\psi_i\mid i \in \Z/(f\Z)\}.\]
Also define
\[\Phi:=\Hom_H(F, \breve{F}).\]
Choose a partition of $\Phi=\Phi_+\sqcup \Phi_-$ such that  
\[\overline{\Phi}_+=\Phi_-.\]
For   $i \in \Z/(f\Z)$, let $\varphi_i$ be the element in $\Phi_+$ such that its restriction to $F_0$ is $\psi_i$.

We denote by  $\Oo_{\breve{H}}$ (resp. $\Oo_{\breve F}$) the ring of integers of $\breve{H}$ (resp. $\breve{F}$). There are decompositions by the Chinese remainder theorem:
\begin{equation}\label{eq:decompositionofOo}
    \Oo_{F_0}\otimes_{\Oo_H} \Oo_{\breve{H}}=\prod_{\psi \in \Psi} \Oo_{\breve H} , \quad \Oo_F\otimes_{\Oo_H} \Oo_{\breve{H}}=\prod_{\psi \in \Psi} \Oo_F\otimes_{\Oo_{F_0},\psi} \Oo_{\breve H}\cong\prod_{\psi \in \Psi} \Oo_{\breve F}.
\end{equation}
The Frobenius $1\otimes\sigma$ is homogeneous and acts simply transitively on the index set.

Let $S$ be an $\Oo_H$-scheme and $\mathfrak{L}$ be a locally free sheaf over $S$ with an $\Oo_H$-action. We say the action is strict if it agrees with the structure map $\Oo_H\rightarrow \Oo_S$.
A strict formal $\Oo_H$-module over $S$ is formal $p$-divisible group over $S$ with an $\Oo_H$-action $\iota:\Oo_H\rightarrow \End(X)$ such that its induced action on $\Lie X$ is strict. Denote by $X^\vee$ the dual of $X$ in the category of strict $\Oo_H$-module, see \cite[Section 11]{mihatschrelative}. We say $X$ is supersingular if its relative Dieudonn\'e module (see \cite[Appendix B.8]{fargues2006isomorphisme}) over $H$ at each geometric point of $S$  has slope $\frac{1}{2}$. 
\begin{definition}
For $S\in \Nilp \Oo_{\breve F}$, a hermitian $\Oo_F$-$\Oo_H$-module over $S$ is a triple $(X,\iota,\lambda)$ where $X$ is a strict formal $\Oo_H$-module together with an action $\iota:\Oo_F\rightarrow \End(X)$ extending the action of $\Oo_H$ and a principal polarization $X\rightarrow X^\vee$ such that 
\[\lambda^{-1}\circ\iota(a)^\vee \circ\lambda =\iota(\bar{a}), \ \forall a\in \Oo_F. \]
Two hermitian $\Oo_F$-$\Oo_H$-modules $(X,\iota,\lambda)$ and $(X',\iota',\lambda')$ are isomorphic (resp. quasi-isogenic) if there is an $\Oo_F$-linear isomorphism (resp. quasi-isogeny) $\varphi:X\rightarrow X'$ such that $\varphi^\vee\circ \lambda'\circ \varphi=\lambda$.
\end{definition}

Let $r,s\in \Z_{\geq 0}$ and set $n:=r+s$. Define the signature function $\Phi\rightarrow \Z_{\geq 0}$ by 
\[r_\varphi=\begin{cases}
r & \text{ if } \varphi=\varphi_0,\\
0 & \text{ if } \varphi\in\Phi_+\backslash\{\varphi_0\},\\
n-r_{\bar{\varphi}} & \text{ if } \varphi\in\Phi_-.
\end{cases}\]
\begin{definition}\label{def:signaturepolys}
For $a\in F$, we define the following polynomial
\begin{align*}
     P_{F/H,(r,s),\varphi_0,\Phi_+}(a;t):=& \prod_{\varphi\in \Phi} (t-\varphi(a))^{r_\varphi}.
\end{align*}
\end{definition}
Let $S\in\Nilp{\Oo_{\breve{F}}}$ and $(\mathfrak{L},\iota)$ a locally free sheaf over $S$ together with an $\Oo_F$ action $\iota$ whose restriction to $\Oo_H$ is strict. Decomposition \eqref{eq:decompositionofOo} induces decomposition of $\mathfrak{L}$:
\begin{equation}\label{eq:decompositionLieX}
    \mathfrak{L}=\bigoplus_{\psi\in\Psi} \mathfrak{L}_{\psi} . 
\end{equation}
\begin{definition}\label{def:signatureconditions}
We say $(\mathfrak{L},\iota)$ satisfies the signature condition $(F/H,(r,s),\varphi_0,\Phi_+)$ if the following conditions are satisfied.
\begin{enumerate}[leftmargin=*, label=({\roman*})]
    \item $\mathrm{charpol}(\iota(a)\mid \mathfrak{L})=P_{F/H,(r,s),\varphi_0,\Phi_+}(a;t)$ for all $a\in \Oo_F$.
    \item $(\iota(a)-a)\mid_{\mathfrak{L}_{\psi_0}}=0$ for all $a\in \Oo_{F_0}$.
    \item For each $\varphi\in \Phi_+$ such that $r_{\varphi}\neq r_{\bar{\varphi}}$, the wedge condition of \cite{P}:
    \[\wedge^{r_{\bar{\varphi}}+1} ((\iota(a)-\varphi(a))\mid \mathfrak{L}_{\psi} )=0, \quad \wedge^{r_{\varphi}+1} ((\iota(a)-\bar{\varphi}(a))\mid \mathfrak{L}_{\psi} )=0\]
     is satisfied for all $a\in \Oo_F$ where $\psi\in \Psi$ is the restriction of $\varphi$ to $F_0$.
\end{enumerate}
\end{definition}
\begin{remark}
When $r_\varphi=n$ or $0$, the condition (iii) above is the same as the banal condition of \cite[Definition 2.60]{liliu2021} or the Eisenstein condition in \cite[Section 2.2]{KRZuniformization}.
\end{remark}
\begin{definition}\label{def:EHmodulewithsignature}
Let $S\in \Nilp \Oo_{\breve{F}}$. Let 
\[\mathfrak{H}_S(F/H,(r,s),\varphi_0,\Phi_+)\]
be the category of supersingular hermitian $\Oo_F$-$\Oo_H$-modules $X$ over $S$ such that the induced $\Oo_F$-action on $\Lie X$ satisfies the signature condition $(F/H,(r,s),\varphi_0,\Phi_+)$.
\end{definition}

\subsection{Comparison theorem}
We will prove the following theorem.
\begin{theorem}\label{thm:comparison}
Assume that $F_0/H$ is unramified.
For $S\in \Nilp\Oo_{\breve F}$, there is an equivalence of categories
\[\mathcal{C}_S:\mathfrak{H}_S(F/H,(r,s),\varphi_0,\Phi_+)\rightarrow \mathfrak{H}_S(F/F_0,(r,s),\varphi_0,\{\varphi_0\})\]
that is compatible with base change.
\end{theorem}
If $S=\Spec R$, we often write $\mathfrak{H}_R$ (resp. $\mathcal{C}_R$) instead of $\mathfrak{H}_S$ (resp. $\mathcal{C}_S$). To prove Theorem \ref{thm:comparison}, we will use the theory of $f$-$\Oo$-displays developed by \cite{ACZ}. We recall some definitions and notations.
For an $\Oo_H$-algebra $R$, let $W_{\Oo_H}(R)$ be the relative Witt ring with respect to a fixed uniformizer of $H$ (see for example \cite[Definition 1.2.2]{FarguesFontaine}). Let $x\mapsto {}^\bF x$ be the Frobenius endomorphism and $x\mapsto {}^\bfV x$ be the Verschiebung. Let $I_{\Oo_H}(R)={}^\bfV W_{\Oo_H}(R)$ and we can define ${}^{\bfV^{-1}}$ on $I_{\Oo_H}(R)$. For $a\in R$, let $[a]\in W_{\Oo_H}(R)$ be its Teichm\"uller representative. 

Let $\hat{\psi}_i$ be the composition of $\psi_i$ with the Cartier morphism $\Oo_{\breve{H}}\rightarrow W_{\Oo_H}(\Oo_{\breve H})$. For $i\in \Z/(f\Z)$, let $\epsilon_i$ be the unique unit in $W_{\Oo_H}(\Oo_{F_0})$ such that ${}^\bfV \epsilon_i=[\psi_i(\pi_0)]-\hat{\psi}_i(\pi_0)$, which exists by \cite[Lemma 2.24]{ACZ}. Following \cite[(2.20)]{liliu2021}, we can define a unit $\mu_\pi\in W_{\Oo_H}(\Oo_{\breve{H}})$ such that
\begin{equation}\label{eq:mu pi}
    \frac{{}^{\bF^f}\mu_\pi}{\mu_\pi}=\prod_{i=1}^{f-1}{}^{\bF^{f-1-i}} \epsilon_i.
\end{equation}

\begin{definition}(\cite[Definition 2.1]{ACZ}). Assume $f\in \Z_{\geq 1}$. 
An $f$-$\Oo_H$-display over $R$ is a quadruple $\cP=(P,Q,\bF,\dot{\bF})$ consisting of the following data: a finitely generated projective $W_{\Oo_H}(R)$-module $P$, a submodule $Q\subset P$, and two ${}^{\bF^f}$-linear maps:
\[\bF:P\rightarrow P,\quad \dot{\bF}:Q\rightarrow P.\]
The following conditions are required. 
\begin{enumerate}[leftmargin=*, label=({\roman*})]
    \item $I_{\Oo_H}(R) P\subseteq Q$ and there is a decomposition of $W_{\Oo_H}(R)$-modules $P=L\oplus T$ such that $Q=L\oplus I_{\Oo_H}(R) T$. Such a decomposition is called a normal decomposition. 
    \item $\dot{\bF}$ is an ${}^{\bF^f}$-linear epimorphism.
    \item For all $x\in P$ and $w\in W_{\Oo_H}(R)$, we have 
    \[\dot{\bF}({}^\bfV wx)={}^{\bF^{f-1}}w \bF(x).\]
\end{enumerate}
We define the Lie algebra of $\cP$ to be $\Lie \cP:=P/Q$. If $f=1$, we simply call $\cP$ an $\Oo_H$-display.
\end{definition}
We refer to \cite[Definition 2.3]{ACZ} for the definition of a nilpotent display and \cite[Section 11]{mihatschrelative} for the notion of polarizations of displays (see also \cite[Section 3]{KRZuniformization}). 
The main result of \cite{ACZ} tells us that there are equivalences of categories:
\[\{\text{nilpotent } f \text{-} \Oo_H \text{-displays over} R\}\rightarrow \{\text{strict formal } \Oo_{F_0}\text{-modules over }R\} \]
where $f=[F_0:H]$, in particular
\[\{\text{nilpotent } \Oo_H \text{-displays over} R\}\rightarrow \{\text{strict formal } \Oo_H\text{-modules over }R\}. \]

\noindent
{\textit{Proof of Theorem} \ref{thm:comparison}}:
\newline
The proof is similar with that of \cite[Proposition 2.62]{liliu2021}.
Assume that $S=\Spec R\in\Nilp\Oo_{\breve F}$. We abuse notation and denote the composition of $\hat{\psi}_i$ with $W_{\Oo_H}(\Oo_{\breve F_0})\rightarrow W_{\Oo_H}(R)$ by $\hat{\psi}_i$ as well.
Then \eqref{eq:decompositionofOo} induces 
\begin{equation}\label{eq:decompositionofOoEW}
    \Oo_F\otimes_{\Oo_H} W_{\Oo_H} (R) =\prod_{\psi\in \Psi} \Oo_F\otimes_{\Oo_{F_0},\hat{\psi}_i} W_{\Oo_H} (R).
\end{equation}

Assume $(X,\iota,\lambda)\in \mathfrak{H}_S(F/H,(r,s),\varphi_0,\Phi_+)$ and $\cP=(P,Q,\bF,\dot{\bF})$ be its associated $\Oo_H$-display. Then $\cP$ has an $\Oo_F$ action (still denoted by $\iota$). Equation
\eqref{eq:decompositionofOoEW} induces the following decomposition
\begin{equation}\label{eq:decompositionofP}
    P=\bigoplus_{\psi \in \Psi} P_\psi,\quad Q=\bigoplus_{\psi \in \Psi} Q_\psi, \text{ with } Q_\psi=P_\psi\cap Q
\end{equation}
where $P_\psi$ has an $\Oo_F\otimes_{\Oo_{F_0},\hat{\psi}_i} W_{\Oo_H} (R)$ action. Then $\bF$ and $\dot{\bF}$ shift the grading on $P$ in the following way.
\[\bF:P_\psi\rightarrow P_{\sigma\circ\psi},\quad \dot{\bF}: Q_\psi\rightarrow P_{\sigma\circ\psi}.\]
As in \cite[Section 11.1]{mihatschrelative}, the principal polarization $\lambda$ is equivalent to a collection of perfect $W_{\Oo_H}(R)$-bilinear skew-symmetric pairings 
\[\{\langle\cdot,\cdot\rangle_\psi :P_\psi\times P_\psi\rightarrow W_{\Oo_H}(R)\mid \psi \in \Psi\}\]
such that $\langle \iota(a) x,y\rangle_\psi=\langle x,\iota(\bar{a}) y\rangle_\psi$ for all $a\in \Oo_F$, $x,y\in P_\psi$ and $ \langle \dot{\bF}x,\dot{\bF}y\rangle_{\sigma\circ\psi}={}^{\bfV^{-1}}\langle x,y\rangle_\psi$ for all $x,y\in Q_{\psi}$.

For $\psi\neq\psi_0$, by \cite[Lemma 2.60]{liliu2021}, the banal signature condition implies
\[Q_\psi=(\pi\otimes 1+1\otimes [\varphi(\pi)])P_\psi+I_{\Oo_H}(R) P_\psi.\]
where $\varphi$ is the element in $\Phi_+$ above $\psi$.
Hence for $\psi\neq\psi_0$, we can define
\[\bF':P_\psi\rightarrow P_{\sigma\circ\psi}:x\mapsto \dot{\bF}((\pi\otimes 1+1\otimes [\varphi(\pi)])x).\]
By loc.cit., $\bF'$ is an ${}^\bF$-linear isomorphism. Now define
\[P^{\mathrm{rel}}=P_{\psi_0}, \quad Q^{\mathrm{rel}}=Q_{\psi_0},\quad \bF^{\mathrm{rel}}=((\bF')^{f-1}\circ \bF)|_{P^{\mathrm{rel}}}, \quad \dot{\bF}^{\mathrm{rel}}=((\bF')^{f-1}\circ \dot{\bF})|_{Q^{\mathrm{rel}}}.\]
Then $\cP^{\mathrm{rel}}:=(P^{\mathrm{rel}},Q^{\mathrm{rel}},\bF^{\mathrm{rel}},\dot{\bF}^{\mathrm{rel}})$ is a $f$-$\Oo_H$-display over $R$. Define
\[\iota^{\mathrm{rel}}:\Oo_F\rightarrow \End(\cP^{\mathrm{rel}})\]
simply by restricting $\iota$ to $P_{\psi_0}$.
Then the signature condition $(F/H,(r,s),\varphi_0,\Phi_+)$ restricted on $P_{\psi_0}$ is exactly the same as as the signature condition $(F/F_0,(r,s),\varphi_0,\{\varphi_0\})$. 
Define 
\[\langle\cdot,\cdot\rangle^{\mathrm{rel}}:=\mu_\pi \langle\cdot,\cdot\rangle|_{P^{\mathrm{rel}}}\]
where $\mu_\pi$ is as in \eqref{eq:mu pi}.
Then $\langle\cdot,\cdot\rangle^{\mathrm{rel}}$ is a perfect $W_{\Oo_H}(R)$-bilinear skew-symmetric pairing such that 
$\langle \iota(a) x,y\rangle^{\mathrm{rel}}=\langle x,\iota(\bar{a}) y\rangle^{\mathrm{rel}}$ for all $a\in \Oo_F$, $x,y\in P^{\mathrm{rel}}$. By the calculation before \cite[Remark 2.61]{liliu2021}, we also have
\[\langle \dot{\bF}^{\mathrm{rel}} x,\dot{\bF}^{\mathrm{rel}} y\rangle^{\mathrm{rel}}={}^{\bF^{f-1}\bfV^{-1}}\langle x,y\rangle^{\mathrm{rel}},\quad \forall x,y \in Q^{\mathrm{rel}}.\]
The form $\langle\cdot,\cdot\rangle^{\mathrm{rel}}$ gives a principal polarization of $\cP^{\mathrm{rel}}$. The pair $(\cP^{\mathrm{rel}},\iota^{\mathrm{rel}})$ together with the polarization gives an object
\[(X,\iota,\lambda)^{\mathrm{rel}}\in \mathfrak{H}_S(F/F_0,(r,s),\varphi_0,\{\varphi_0\}).\]
This is defined to be $\mathcal{C}_S((X,\iota,\lambda))$. The functor $\mathcal{C}_S$ is obviously functorial in $S$. The fact that $\mathcal{C}_S$ is an equivalence of categories can be proved verbatim as that of \cite[Proposition 2.62]{liliu2021}. 
\qedsymbol

\subsection{Comparison of Rapoport-Zink spaces}\label{subsec:comparisonofRZspaces}
Fix a triple 
\[(\bX^{F/H}, \iota^{F/H}_\bX,\lambda^{F/H}_\bX)\in \mathfrak{H}_{\bar\bk}(F/H,(r,s),\varphi_0,\Phi_+).\]
We essentially only have one or two such choices up to isogeny according to $n$ being odd or even, see Remark \ref{rmk:twoframingobjects} below.
\begin{definition}\label{def:RZspace}
Let $\mathcal{N}_{(r,s)}^{F/H}$ be the functor which associates to $S\in\Nilp\Oo_{\breve F}$ the set of isomorphism classes of quadruples $(X,\iota,\lambda,\varrho)$ where
\begin{enumerate}[leftmargin=*, label=({\roman*})]
   \item $(X,\iota,\lambda)\in  \mathfrak{H}_{S}(F/H,(r,s),\varphi_0,\Phi_+)$,
   \item $\varrho: X\times_S \bar{S} \rightarrow \bX^{F/H}\times_{\Spec \bar\bk}\bar{S} $ 
   is a $\Oo_F$-linear quasi-isogeny of height $0$ such that $\lambda$ and $\varrho^* (\lambda^{F/H}_\bX)$ differ locally on $\bar{S}$ by a factor in $\Oo_{H}^{\times}$. 
\end{enumerate}
An isomorphism between two such quardruples $(X,\iota,\lambda,\varrho)$ and $(X',\iota',\lambda',\varrho')$ is given by an $\Oo_F$-linear isomorphism $\alpha:X\rightarrow X'$ such that $\varrho'\circ(\alpha\times_S \bar{S})=\varrho $ and $\alpha^*(\lambda')$ is an $\Oo_{H}^\times$ multiple of $\lambda$.
\end{definition}
\begin{remark}
In the definition of $\mathcal{N}_{(r,s)}^{F/H}$, we can replace condition (ii) by the condition that $\varrho$ is a $\Oo_F$-linear quasi-isogeny of height $0$ such that $\lambda=\varrho^* (\lambda_\bX)$. The resulting functor is isomorphic to the original one as $(X,\iota,\lambda,\varrho)$ and $(X,\iota,a\lambda,\varrho)$ are isomorphic in $\mathcal{N}_{(r,s)}^{F/F_0}$ for $a\in \Oo_{H}^\times$.
\end{remark}
By \cite[Chapter 3]{RZ},  $\mathcal{N}_{(r,s)}^{F/H}$ is representable by a formal scheme locally of finite type over $\Spf \Oo_{\breve F}$. 
\begin{theorem}\label{thm:comparisonRZspaces}
Assume $F_0/H$ is unramified and the framing object $(\bX^{F/F_0}, \iota^{F/F_0}_\bX,\lambda^{F/F_0}_\bX)$ used in the definition of $\cN^{F/F_0}_{(r,s)}$ is isomorphic to $\mathcal{C}_{\bar\bk}((\bX^{F/H}, \iota^{F/H}_\bX,\lambda^{F/H}_\bX))$. Then there is an isomorphism 
\[\mathfrak{C}:\mathcal{N}_{(r,s)}^{F/H} \cong \mathcal{N}_{(r,s)}^{F/F_0}.\]
\end{theorem}
\begin{proof}
This is a consequence of Theorem \ref{thm:comparison}.
\end{proof}

\subsection{The relative Rapoport-Zink space}\label{subsec:relativeRZ}
In this subsection we assume $H=F_0$. We simply denote $\mathcal{N}_{(r,s)}^{F/F_0}$ by $\mathcal{N}_{(r,s)}$ and $\mathfrak{H}_S(F/F_0,(r,s),\varphi_0,\{\varphi_0\})$ by $\mathfrak{H}_S(r,s)$.
We recall some background information on $\mathcal{N}_{(1,n-1)}$ from \cite{RTW}. Although \cite{RTW} works on the category of $p$-divisible groups, their arguments and results easily extend to the category of strict formal $\Oo_{F_0}$-modules using relative Dieudonn\'e theory. 
\begin{proposition}(\cite[Proposition 2.1]{RTW})
The functor $\cN_{(1,n-1)}$ is representable by a separated formal scheme $\cN_{(1,n-1)}$, locally of finite type and flat over $\Spf \Oo_{\breve{F}}$. It is formally smooth over $\Spf \Oo_{\breve{F}}$ in all points of the special fiber except the superspecial points. Here a point $z\in\cN_{(1,n-1)}(\bk)$ is superspecial if $\Lie(\iota(\pi))=0$ where $(X,\iota,\lambda,\varrho )$ is the pullback of the universal object of $\cN_{(1,n-1)}$ to $z$. The superspecial points form an isolated set of points. 
\end{proposition}
For the signature $(0,1)$ we know that $\cN_{(0,1)}\cong \Spf \Oo_{\breve{F}}$ and has a universal formal $\Oo_F$-module $\cY$ (the canonical lifting of $\bY$ in the sense of \cite{G}) over it. 
\begin{remark}\label{rmk:cNnotation}
The formal scheme $\cN_{(1,n-1)}$ is denoted as $\cN^0$ in \cite{RTW}. In the rest of this section and Section \ref{sec:localspecialcycle} we often simply write $\cN$ for $\cN_{(1,n-1)}$ if the context is clear. 
\end{remark}

Let $F^u$ be the unique unramified quadratic extension of $F_0$ in $\breve{F}_0$ where $\breve{F}_0$ is the completion of the maximal unramified extension of $F_0$ in $\breve{F}$. Let $\sigma\in \mathrm{Gal}(\breve{F}_0/F_0)$ be the Frobenius element.
For a formal $\Oo_{F_0}$-module, we denote by $M(X)$ the relative Dieudonn\'e module of $X$. When $X$ has $F_0$-height $n$ and dimension $n$ over $\bar\bk$, $M(X)$ is a free $\Oo_{\breve{F}_0}$-module of rank $2n$ with a $\sigma$-linear operator $\bF$ and a $\sigma^{-1}$-linear operator $\bfV$ such that $\bfV \bF=\bF\bfV=\pi_0$. Denote by $\mathbb{E}=\breve{F}_0[\bF,\bfV]$ the rational Cartier ring.

Fix a framing object $(\bX,\iota_\bX,\lambda_\bX)\in \mathfrak{H}_{\bar\bk}(1,n-1)$. Let $N:=M(\bX)\otimes_\Z \Q$ be the rational relative Dieudonn\'e module of $\bX$. Then $N$ has a skew-symmetric $\breve{F}_0$-bilinear form $\langle\cdot,\cdot\rangle$ induced by $\lambda_\bX$ such that for any $x,y\in N$ we have 
\begin{align*}
    \langle\bF x,y\rangle=& \langle x,\bfV y\rangle^\sigma, \\
    \langle\iota(a)x,y\rangle=&\langle x,\iota(\bar{a})y\rangle, a \in F.
\end{align*}
We simply denote by $\pi$ the induced action of $\iota_{\bX}(\pi)$ on $N$. 
Define a $\sigma$-linear operator 
\begin{equation}\label{eq:tau}
    \tau=\pi \bfV^{-1}=\pi^{-1} F
\end{equation}
on $N$. Set $C=N^\tau$ (the set of $\tau$-fixed points in $N$), then we obtain a $n$-dimensional $F$-vector space with an isomorphism
\[C\otimes_{F} \breve{F} \simeq N.\]
For $x,y \in C$, we have
\begin{align*}
    \langle x,y\rangle =& \langle \tau(x),\tau(y)\rangle\\
    =&\langle \pi^{-1} \bF x, \pi \bfV^{-1} y\rangle\\
    =&-\langle \bF x, \bfV^{-1} y \rangle \\
    =&- \langle x,y\rangle^\sigma
\end{align*}
Choose $\delta\in F^u \backslash F_0$ such that $\delta^2\in \Oo_{F_0}^\times$. Define a form $(\cdot,\cdot)$ on $N$ by
\begin{equation}\label{eq:definitionof(,)}
    (x,y)=\delta(\langle\pi x,y\rangle+\pi\langle x,y\rangle)
\end{equation}
for all $x,y\in C$.
Then $(\cdot,\cdot)$ is Hermitian with values in $F$ when restricted on $C$ and
\begin{equation}\label{eq:recover <,>}
   \langle x,y\rangle=\frac{1}{2\delta} \mathrm{tr}_{F/F_0}(\pi^{-1}(x,y)), \forall x,y \in C. 
\end{equation}

\begin{remark}\label{rmk:latticeofY}
There is a unique object $(\bY,\iota_\bY,\lambda_\bY)\in \mathfrak{H}_{\bar\bk}(0,1)$ up to isomorphism.
We want to describe $M(\bY)$ explicitly. As an $\Oo_{F_0}$-lattice, it is of rank 2. We can choose a basis $\{e_1,e_2\}$ such that $\bF e_1=e_2,\bF e_2=\pi_0 e_1,\bfV e_1=e_2,\bfV e_2=\pi_0 e_1$ and $\langle e_1,e_2\rangle =\delta$. With respect to this basis,
$\mathrm{End}^0(\bY)=\mathrm{End}_{\mathbb{E}}(N)$ is of the form
\[\left\{\left(\begin{array}{cc}
    a & b \pi_0  \\
    b^\sigma & a^\sigma
\end{array}\right)\mid a,b \in F^u\right\},\]
which is the quaternion algebra $\mathbb{H}$ over $F_0$. By changing basis using elements in $\mathbb{H}\cap \mathrm{SL}_2(F_0)$ we can assume $F,V$ are of the same matrix form as before and 
\[\pi=\left(\begin{array}{cc}
    0 & \pi_0  \\
    1 & 0
\end{array}\right).\]
Thus $\tau$ is the diagonal matrix $\mathrm{diag}\{1,1\}$ and fixes the $F_0$-vector space $\mathrm{span}_{F_0}\{ e_1,e_2\}$. We have $(e_1,e_1)=-\delta^2$.
As $\Oo_F$ is a DVR and $N^\tau$ is a one dimensional $F$-space, $\mathrm{span}_{\Oo_F}\{e_1\}$ is the unique self-dual $\Oo_F$-lattice w.r.t. $(\cdot,\cdot)$. Let $\varrho_{\bY}$ be the identity of $\bY$, then $(\bY,\iota_\bY, \lambda_\bY,\varrho_\bY)$ is the unique closed point of $\cN_{(0,1)}(\bar\bk)$.
\end{remark}
\begin{remark}\label{rmk:twoframingobjects} 
By \cite[Remark 4.2]{RTW}, when $n$ is odd (resp. even) there is a unique (resp. exactly two) object $(\bX,\iota_\bX,\lambda_{\bX})\in \mathfrak{H}_{\bar\bk}(1,n-1)$ up to isogenies that preserves the $\lambda_\bX$ by a factor in $\Oo_{F_0}^\times$. These are the framing objects in the definition of $\cN_{(1,n-1)}$. This matches the number of similarity classes of Hermitian forms over local fields. 

When $n$ is odd, we simply take $(\bX,\iota_\bX,\lambda_\bX):=(\bY,\iota_\bY,\lambda_\bY)^n$ where $(\bY,\iota_\bY,\lambda_\bY)$ is defined in the previous remark.
When $n$ is even, we again define $\bX:=\bY^n$ with the diagonal action $\iota_{\bX}$ by $\Oo_F$. There are two choices of polarizations.
The first one $\lambda_{\bX}^+\in \End^0(\bX)\cong M_{n}(\H)$ is given by the anti-diagonal matrix with $1$'s on the anti-diagonal. The second one $\lambda_{\bX}^-$ is defined by the diagonal matrix $\mathrm{diag}(1,\ldots,1,u_1,u_2)$ where $u_1,u_2\in \Oo_{F_0}^\times$ and $-u_1 u_2\notin \mathrm{Nm}_{F/F_0}(F^\times)$. 
\end{remark}

For two $\Oo_F$-lattices $\Lambda$, $\Lambda'$ of $C$, we use the notation $\Lambda {\subset}^\ell \Lambda'$ to stand for the situation when $\pi \Lambda'\subseteq \Lambda \subseteq \Lambda'$ and $ \mathrm{dim}_{\bk}(\Lambda'/ \Lambda)=\ell$. Define
\begin{equation}\label{eq:LambdasharpandLambdavee}
    \Lambda^\sharp:=\{x\in C \mid (x,\Lambda)\subseteq \Oo_F\}, \quad \Lambda^\vee:=\{x\in C\mid \delta \langle x,\Lambda\rangle\subseteq \Oo_{F_0}\}.
\end{equation}
Similarly for an $\Oo_{\breve F}$-lattice $M\subset N$, define
\begin{equation}\label{eq:MsharpandMvee}
    M^\sharp:=\{x\in N \mid (x,M)\subseteq \Oo_{\breve{F}}\}, \quad M^\vee:=\{x\in N\mid \langle x,M\rangle\subseteq \Oo_{\breve{F}_0}\}.
\end{equation}
Then by \eqref{eq:definitionof(,)} and \eqref{eq:recover <,>}, $\Lambda^\sharp=\Lambda^\vee$. Similarly $M^\sharp=M^\vee$.
\begin{proposition}\label{cV}(\cite[Proposition 2.4]{RTW})
Define the following set of $\Oo_{\breve F}$-lattices 
\[\mathcal{V}:=\{ M\subseteq N \mid M^\sharp=M,\\
\pi \tau (M)\subseteq M \subseteq \pi^{-1} \tau(M),M\subset^{\leq 1} (M+\tau(M))\},\]
Then the map 
\[(X,\iota,\lambda,\varrho)\mapsto \varrho(M(X))\subset N\]
defines a bijection from $\cN(\bar\bk)$ to $\cV$.
\end{proposition}

A vertex lattice in $C$ is an $\Oo_F$-lattice $\Lambda\subset C$ such that $\pi \Lambda \subseteq \Lambda^\sharp \subseteq \Lambda$. We denote the dimension of the $\bk$-vector space $\Lambda/\Lambda^\sharp$ by $t(\Lambda)$, and call it the type of $\Lambda$. It is an even integer (see \cite[Lemma 3.2]{RTW}). 
\begin{lemma}(\cite[Proposition 4.1]{RTW})\label{lem:Zinklemma}
$\forall M\in \cV$, there is a unique minimal vertex lattice $\Lambda(M)$ such that $M\subseteq \Lambda(M)\otimes_{\Oo_F} \Oo_{\breve{F}}$.
\end{lemma}
Define
\[\cV(\Lambda):=\{M\in \cV\mid M\subseteq \Lambda\otimes_{\Oo_F} \Oo_{\breve{F}}\},\]
and 
\[\cV^o(\Lambda):=\{M\in \cV\mid \Lambda(M)=\Lambda\}.\]
Then apparently $ \cV^o(\Lambda)\subseteq \cV(\Lambda)$. The following theorem summarizes what we need from \cite[Section 6]{RTW}, in particular Theorem 6.10 of loc. cit..
\begin{theorem}\label{thm:incidencerelation}
We have the following facts.
\begin{enumerate}[leftmargin=*, label=({\roman*})]
\item For two vertex lattices $\Lambda_1$ and $\Lambda_2$
\[\cV(\Lambda_1)\subseteq \cV(\Lambda_2) \Leftrightarrow \Lambda_1 \subseteq \Lambda_2. \]
If $\Lambda_1\cap \Lambda_2$ is a vertex lattice, then 
\[\cV(\Lambda_1 \cap \Lambda_2)=\cV(\Lambda_1)\cap \cV(\Lambda_2),\]
otherwise $\cV(\Lambda_1)\cap \cV(\Lambda_2)=\emptyset $. 
\item For each vertex lattice $\Lambda$, there exist a reduced projective variety $\cN^o_{\Lambda}$ over $\Spec \bar\bk$ such that 
\[\cN^o_\Lambda(\bar\bk)=\cV^o(\Lambda).\]
The closure of any $\cN_{\Lambda}^o$ in $\cN_{\mathrm{red}}$ is given by 
\[\cN_{\Lambda}:=\biguplus_{\Lambda'\subseteq \Lambda} {\cN}_{\Lambda'}^o,\]
where the union is taken over vertex lattices $\Lambda'$ included in $\Lambda$.
$\cN_{\Lambda}$ is a projective variety of dimension $t(\Lambda)/2$. Its set of $\bar\bk$ points is $\cV(\Lambda)$. The inclusion of points $\cV(\Lambda_1)\subseteq \cV(\Lambda_2)$ in (i) is induced by a closed embedding $\cN_{\Lambda_1} \rightarrow \cN_{\Lambda_2}$.
\item There is a stratification of the reduced locus of $\cN$ given by
\[\cN_{\mathrm{red}}=\biguplus_{\Lambda} {\cN}_{\Lambda}^o\]
where the union is over all vertex lattices. $\cN(\bar \bk)$ is nonempty for all $n\geq 1$.
\end{enumerate}
\end{theorem}

\section{Special cycles on Rapoport-Zink spaces }\label{sec:localspecialcycle}
In this section, we define special cycles on $\cN^{F/H}_{(1,n-1)}$. We then state our main results on the support of these cycles. First we need some background information on Hermitian lattices.
\subsection{Hermitian lattices and Jordan splitting}\label{subsec:Hermitianlattices}
We use $\obot$ to denote direct sum of mutually orthogonal spaces. In particular, we use 
\[(\alpha_1)\obot \cdots \obot (\alpha_n)\]
to denote the $n$-dimensional $F$ vector space (or $\Oo_F$-lattice depending on the context) with a Hermitian form given by a diagonal matrix $\mathrm{diag}\{\alpha_1,\ldots,\alpha_n\}$ with respect to an orthogonal basis. We also use $H(i)$ to denote the hyperbolic plane which is the lattice of rank $2$ with Hermitian form given by the matrix
\[\left(\begin{array}{cc}
    0 & \pi^i  \\
    (-\pi)^i & 0
\end{array}\right)\]
with respect to a certain basis.

For a Hermitian lattice $L$ with Hermitian form $(\cdot,\cdot)$, define $sL$ to be  $\mathrm{min}\{\mathrm{val}_{\pi}(x, y)\mid x,y\in L\}$ where $\mathrm{val}_{\pi}$ is normalized such that $\mathrm{val}_{\pi}(\pi)=1$. We say $x\in L$ is maximal if $x$ is not in $\pi L$. We say $L$ is $\pi^i$-modular if $(x, L)=\pi^i \Oo_F$ for every maximal vector $x$ in $L$.

Any Hermitian lattice $L$ has a Jordan splitting 
\begin{equation}\label{Jordansplitting}
    L=\obot_{\lambda \in \Z\cup \{\infty\}} L_\lambda
\end{equation}
where $L_\lambda$ is $\pi^\lambda$-modular and $L_\infty$ is defined to be the radical of $L$. Any two Jordan splitting of $L$ have the same invariants, see \cite[Page 449]{J}. 

\begin{proposition}(\cite[Proposition 8.1]{J})\label{prop:modularlattice}
Let $L$ be a $\pi^i$-modular lattice of rank $n$. Then 
\begin{enumerate}
    \item $L\simeq (\pi_0^\frac{i}{2})\obot (\pi_0^\frac{i}{2}) \obot \cdots \obot (\pi_0^{-\frac{(n-1)i}{2}} \mathrm{det}(L)) $ if $i$ is even.\\
    \item $L\simeq H(i)\obot H(i)\obot \cdots \obot H(i)$ if $i$ is odd.
\end{enumerate}
In particular, when $i$ is odd $L$ must have even rank. 
\end{proposition}
For a sub $\Oo_F$-module $L$ in a Hermitian $F$-vector space $V$, define
\begin{equation}\label{eq:L V sharp}
    L^\sharp_V:=\{x\in V\mid (x,y)\in \Oo_F , \forall y \in L\}.
\end{equation}
When $V=L\otimes_\Z \Q$, we simply denote $L_V^\sharp$ by $L^\sharp$.
We will use the following basic lemmas throughout the paper, sometimes without explicitly referring to them.
\begin{lemma}\label{lem:L dual}
Assume $L$ and $L'$ are  $\Oo_F$-submodules inside a Hermitian $F$-vector space $V$. Then
\begin{enumerate}
    \item $(L+L')_V^\sharp=L_V^\sharp \cap (L')_V^\sharp$;
    \item $(L_V^\sharp)_V^\sharp=L$.
\end{enumerate}
\end{lemma}
\begin{proof}
(1) follows from the definition of $L_V^\sharp$. (2) can be proved by using the Jordan splitting of $L$.
\end{proof}
\begin{lemma}\label{directsummand}
Assume $L'$ is a sub $\Oo_F$-module of $L$ such that $L'$ is $\pi^s$-modular with $s=sL$. Then $L=L'\obot (L')^\bot$ where $(L')^\bot$ is the perpendicular complement of $L'$ in $L$. 
\end{lemma}
\begin{proof}
This is a direct consequence of \cite[Proposition 4.2]{J}.
\end{proof}

\subsection{Special cycles}\label{subsec:localspecialcycle}
For a moment, we go back to the setting of Section \ref{subsec:comparisonofRZspaces}.
Let $(\bX^{F/H},\iota^{F/H}_\bX,\lambda^{F/H}_\bX)$ (resp. $(\bY^{F/H},\iota^{F/H}_\bY,\lambda^{F/H}_\bY)$) be the framing object of $\cN^{F/H}_{(1,n-1)}$ (resp. $\cN^{F/H}_{(0,1)}$).
Define the space of special homomorphisms to be the $F$-vector space
\begin{equation}\label{eq:bV}
    \bV^{F/H}:=\mathrm{Hom}_{\Oo_F}(\bY^{F/H},\bX^{F/H})\otimes_\Z \Q
\end{equation}
Define a Hermitian form $h^{F/H}(\cdot,\cdot)$ on $\bV^{F/H}$ such that for any $x,y\in \bV^{F/H}$ we have 
\begin{equation}\label{h(,)}
    h^{F/H}(x,y)=(\lambda^{F/H}_\bY)^{-1} \circ y^\vee \circ\lambda^{F/H}_\bX \circ x\in \mathrm{End}_{\Oo_F}(\bY^{F/H})\otimes \Q \xrightarrow[\sim]{(\iota_\bY^{F/H})^{-1}} F
\end{equation}
as in \cite[Equation (3.1)]{KR1} where $y^\vee$ is the dual quasi-isogeny of $y$.

\begin{definition}\label{def:localspecialcycle}
For an $\Oo_F$-lattice $\bL\subset \bV^{F/H}$, the special cycle $\cZ(\bL)$ is the subfunctor of $\cN^{F/H}_{(0,1)}\times_{\Spf \Oo_{\breve F}} \cN^{F/H}_{(1,n-1)}$ whose $S$-points is the set of isomorphism classes of tuples
\[\xi=(Y,\iota,\lambda_{Y}, \varrho_{Y}, {X},\iota,\lambda_{ X},\varrho_{X})\in \cN^{F/H}_{(0,1)}\times_{\Spf \Oo_{\breve F}} \cN^{F/H}_{(1,n-1)}(S)\]
such that for any $\bx\in \bL$ the quasi-homomorphism 
    \[\varrho^{-1}_{X}\circ \bx\circ \varrho_{ Y}: Y\times_S \bar{S}\rightarrow X\times_S \bar{S}\]
deforms to a homomorphism from $Y$ to $X$. If $\bL$ is spanned by $\bx\in \bV^m$, we also denote $\cZ(\bL)$ by $\cZ(\bx)$.
\end{definition}
By Grothendieck-Messing theory, $\cZ(\bL)$ is a closed sub formal scheme in $\cN^{F/H}_{(0,1)}\times_{\Spf \Oo_{\breve F}} \cN^{F/H}_{(1,n-1)}$. 
\begin{proposition}\label{prop:comparisonofspecialcycles}
Keep the same assumption as Theorem \ref{thm:comparisonRZspaces}.
The functor $\mathcal{C}_{\bar\bk}$ in Theorem \ref{thm:comparison} induces an isomorphism (denoted by the same notation) $\mathcal{C}_{\bar\bk}:\bV^{F/H}\rightarrow \bV^{F/F_0}$ of Hermitiam vector spaces over $F$. Moreover for lattice $\bL\in \bV^{F/H}$, the functor $\mathfrak{C}$ in Theorem \ref{thm:comparisonRZspaces} induces an isomorphism of formal schemes:
\[\cZ(\bL)\rightarrow\cZ(\mathcal{C}_{\bar\bk}(\bL)).\]
\end{proposition}
\begin{proof}
This follows directly from Theorem \ref{thm:comparison}. One can compare our result with \cite[Remark 4.4]{mihatschrelative}.
\end{proof}

\subsection{Special cycles on Relative Rapoport-Zink spaces}
By Proposition \ref{prop:comparisonofspecialcycles}, we can without loss of generality assume that $H=F_0$. In this case we drop the superscript ${}^{F/F_0}$ over $\bX$, $\bY$, $\bV$ and $h$ etc.
For $x,y\in \bV$ we abuse notation and denote the induced map between the corresponding relative Dieudonn\'e modules still by $x,y$. As in Section \ref{subsec:relativeRZ}, we denote $\cN_{(1,n-1)}$ simply by $\cN$.
\begin{lemma}\label{lem:equivalenceofHermitianforms}
We have
\[h(x,y)(e_1,e_1)_\bY=(x(e_1),y(e_1))_\bX\]
where $e_1$ is as in Remark \ref{rmk:latticeofY} and $(\cdot,\cdot)_\bX,(\cdot,\cdot)_\bY$ are defined as in equation \eqref{eq:definitionof(,)} for the rational relative Dieudonn\'e module of $\bX$ and $\bY$ respectively.
\end{lemma}
\begin{proof}
We claim that $\lambda_\bY^{-1}\circ y^\vee\circ \lambda_\bX$ agrees with $y^*$ which is the adjoint operator of $y$ on $\mathrm{Hom}_{\mathbb{E}}(M(\bY)\otimes \Q,M(\bX)\otimes \Q)$ w.r.t. $\langle\cdot,\cdot\rangle_\bX$ and $\langle\cdot,\cdot\rangle_\bY$. In fact $\langle\cdot,\cdot\rangle_\bX$ is defined by $e\langle\cdot,\lambda_\bX \circ\cdot\rangle_\bX$ where $e\langle\cdot,\cdot\rangle_\bX$ is the pairing between $M(\bX)\otimes \Q$ and $M(\bX^\vee)\otimes \Q$, similarly for $\langle\cdot,\cdot\rangle_\bY$. Hence 
\begin{align*}
    \langle y(n),m\rangle _\bX=e\langle y(n),\lambda_\bX (m)\rangle_\bX=e\langle n, y^\vee\lambda_\bX(m)\rangle_\bY=\langle n, \lambda_\bY^{-1}y^\vee\lambda_\bX(m)\rangle _\bY,
\end{align*}
for all $n\in M(\bY)\otimes \Q$ and $m\in M(\bX)\otimes \Q$. This proves the claim. Hence
\begin{align*}
    (x(e_1),y(e_1))_\bX=& \langle\pi x(e_1),y(e_1)\rangle _\bX+\pi \langle x(e_1),y(e_1)\rangle _\bX \\
    =& \langle y^* \pi x(e_1),e_1\rangle _\bY+\pi \langle y^*  x(e_1),e_1\rangle _\bY  \\
    =& \langle \pi y^*  x(e_1),e_1\rangle _\bY+\pi \langle y^*  x(e_1),e_1\rangle _\bY  \\
    =& \langle \pi h(x,y) e_1,e_1\rangle _\bY+  \pi \langle h(x,y) e_1,e_1\rangle _\bY  \\
    =& h(x,y)(e_1,e_1)_\bY.
\end{align*}
This proves the lemma.
\end{proof}

Now assume $\bL\subset \bV$ is an $\Oo_F$-lattice and define
\begin{equation}\label{eq:L}
   L=\{\bx(e_1)\mid \bx\in \bL\}. 
\end{equation}
Then $L$ is an $\Oo_F$-lattice in $C$ with the same rank as $\bL$ and is similar to $\bL$ as a Hermitian lattice by Lemma \ref{lem:equivalenceofHermitianforms}.

\begin{definition}
Define $\mathrm{Vert}(L)$ to be the set of vertex lattices $\Lambda$ such that $L\subseteq \Lambda^\sharp$. We also define
\begin{equation}\label{eq:cWL}
  \cW(L):=\{M\in \cV\mid L \subseteq M\}\subset \cV=\cN(\bar\bk).
\end{equation}
\end{definition}
\begin{proposition}\label{prop:startinglemma}
For an $\Oo_F$-lattice $\bL\subset \bV$ , define $L$ as in equation \eqref{eq:L}.
The set of $\bar\bk$ points of the special cycle $\cZ(\bL)$ is $\cW(L)$. Moreover we have
\begin{equation}\label{eq:Z L union N Lambda}
    \cZ(\bL)_{\mathrm{red}}=\bigcup_{\Lambda\in \mathrm{Vert}(L)} \cN_\Lambda.
\end{equation}
\end{proposition}
\begin{proof}
Assume that $(X,\iota,\lambda,\varrho)$ is a point in $\cN(\bar\bk)$ and $M:=\varrho(M(X))\in \cV$ as in Proposition \ref{cV}. 
By Dieudonn\'e theory, for any $\bx\in \bL$, $\varrho^{-1}\circ\bx$ is a homomorphism from $\bY$ to $X$ if and only if $\varrho^{-1}\circ\bx(M(\bY))\subseteq M(X)$, if and only if $\bx(M(\bY))\subseteq M$. We know that $M(\bY)=\mathrm{span}_{\Oo_{\breve{F}}}\{e_1\}$. Hence the set of $\bar\bk$ points of the special cycle $\cZ(\bL)$ is $\cW(L)$. 

To prove \eqref{eq:Z L union N Lambda}, since both sides of the equation are reduced, it suffices to check it on the $\bar\bk$-points, namely,
\[\cW(L)=\bigcup_{\Lambda\in \mathrm{Vert}(L)} \cV(\Lambda).\]
Let $M\in \cV$ and suppose $\Lambda=\Lambda(M)$ as in Lemma \ref{lem:Zinklemma}. Then
\begin{align*}
    L\subseteq M \Leftrightarrow & M\subseteq (L_C^\sharp)\otimes_{\Oo_F} \Oo_{\breve F}\ \text{ as}\ M=M^\sharp \  (\text{recall \eqref{eq:L V sharp}}),\\
    \Leftrightarrow & \Lambda\subseteq (L_C^\sharp)\otimes_{\Oo_F} \Oo_{\breve F} \ \text{ as } L^\sharp \text{ is } \tau \text{-invariant}, \\
    \Leftrightarrow& L\subseteq \Lambda^\sharp \text{ by Lemma \ref{lem:L dual}}.
\end{align*}
This in fact shows that 
\[M\in \cW(L)\Leftrightarrow \cV^o(\Lambda)\subseteq \cW(L).\]
Hence 
\[\cW(L)=\bigcup_{\Lambda\in \mathrm{Vert}(L)} \cV^o(\Lambda)=\bigcup_{\Lambda\in \mathrm{Vert}(L)} \cV(\Lambda)\]
where the last equality follows from (i) and (ii) of Theorem \ref{thm:incidencerelation}. This finishes the proof of the proposition.
\end{proof}

\begin{corollary}\label{cor:Tisintegral}
If $\cZ(\bL)(\bar\bk)$ is non-empty, then $\bL$ is integral, i.e.  $h(\bx,\by)\in \Oo_F$ for any $\bx,\by\in \bL$.
\end{corollary}
\begin{proof}
By Proposition \ref{prop:startinglemma}, there exists an $M\in \cV$ such that $L\subseteq M$. By Lemma \ref{lem:equivalenceofHermitianforms}, we have
\[h(\bx,\by)=\frac{(\bx(e_1),\by(e_1))_\bX}{(e_1,e_1)_\bY}.\]
Since $M=M^\sharp$, we know $(\bx(e_1),\by(e_1))_\bX \in (M,M)_\bX =\Oo_{\breve F}$. Also notice that $(e_1,e_1)_\bY \in \Oo_{F_0}^\times$ by construction. The lemma follows. 
\end{proof}

From now on we assume that $L$ (or $\bL$ equivalently) has rank $n$.
Take a Jordan decomposition of $\bL$ as in \eqref{Jordansplitting}.
By Corollary \ref{cor:Tisintegral}, $\lambda\geq 0$ for all $\lambda$ such that $L_\lambda\neq\{0\}$. We define 
\begin{equation}\label{defineLt}
    \bL_{\geq t}=\obot_{\lambda\geq t} \bL_\lambda,
\end{equation}
and
\begin{equation}\label{eq:m(L)}
    m(\bL)=\mathrm{rank}_{\Oo_F}(\bL_{\geq1}).
\end{equation}
Also define
\begin{equation}\label{eq:definenodd}
    n_{\mathrm{odd}}=\sum_{\lambda\geq 3,\lambda \text{ is odd}}\mathrm{rank}_{\Oo_F}(\bL_\lambda), \text{ and }
    n_{\mathrm{even}}=\sum_{\lambda\geq 2,\lambda \text{ is even}}\mathrm{rank}_{\Oo_F}(\bL_\lambda).
\end{equation}

We say a Hermitian $F$-vector space $V$ of even dimension is split if it is isomorphic to sum of copies of $H(0)\otimes_\Z \Q$. Equivalent it is split if and only if $(-1)^{n(n-1)/2} \mathrm{det} (V)\in \mathrm{Nm}_{F/F_0} F^\times$ where $n$ is the dimension of $V$.
The following theorem is the analog of \cite[Theorem 4.2]{KR1}.
\begin{theorem}\label{thmA}
Assume that $\bL\subset \bV$ has rank $n$ and is integral. Define $L$ as in \eqref{eq:L}. Then
    \[\cZ(\bL)_{\mathrm{red}}=\bigcup_{\{\Lambda\in \mathrm{Vert}(L)\mid t(\Lambda)=\mathfrak{d}(\bL)\}}\cN_\Lambda\]
    where
    \[\mathfrak{d}(\bL):=\left\{\begin{array}{cc}
      m(\bL)-1   & \text{if $m(\bL)$ is odd} \\
      m(\bL)  & \text{if $m(\bL)$ is even and $\bL_{\geq1}\otimes_{\Z} \Q$ is split}\\
      m(\bL)-2  & \text{if $m(\bL)$ is even and $\bL_{\geq1}\otimes_{\Z} \Q$ is non-split}.\\
    \end{array}\right.\]
\end{theorem}
We postpone the proof of Theorem \ref{thmA} and \ref{thmB} below to the Section \ref{subsec:proofofthmA} and \ref{subsec:proofofthmB} respectively.

\begin{corollary}\label{dimensionofcW}
If it is non-empty, $\cZ(\bL)$ is a variety of pure dimension $\frac{1}{2}\mathfrak{d}(\bL)$.
\end{corollary}
\begin{proof}
The corollary follows from Theorem \ref{thmA} and Theorem \ref{thm:incidencerelation}.
\end{proof}

The following theorem is the analog of \cite[Theorem 4.5]{KR1}.
\begin{theorem}\label{thmB}
Make the same assumption as Theorem \ref{thmA}.
$\cZ(\bL)_{\mathrm{red}}=\cN_\Lambda$ for a unique vertex lattice $\Lambda$ if and only if the following two conditions are satisfied simultaneously
\begin{enumerate}
    \item $n_{\mathrm{odd}}=0$. \
    \item $n_{\mathrm{even}}\leq 1$ or $n_{\mathrm{even}}=2$ and $\bL_{\geq 2}\otimes_\Z \Q$ is non-split.     
\end{enumerate}
\end{theorem}

\begin{corollary}
$\cZ(\bL)_{\mathrm{red}}$ is an irreducible variety if and only if condition (1) and (2) in Theorem \ref{thmB} are satisfied. 
\end{corollary}
\begin{proof}
By Proposition \ref{prop:startinglemma} and Theorem \ref{thm:incidencerelation}, $\cZ(\bL)_{\mathrm{red}}$ is an irreducible variety if and only if $\cZ(\bL)_{\mathrm{red}}=\cN_\Lambda$ for a unique vertex lattice $\Lambda$. The corollary now follows from Theorem \ref{thmB}.
\end{proof}

\begin{corollary}
The variety $\cZ(\bL)_{\mathrm{red}}$ is zero dimensional if and only if the following conditions are satisfied 
\begin{enumerate}
    \item $n_{\mathrm{odd}}=0$. \ 
    \item $\mathrm{rank}_{\Oo_F}(\bL_1)=0$.\
    \item $n_{\mathrm{even}}\leq 1$ or $n_{\mathrm{even}}=2$ and $\bL_{\geq 2}\otimes_\Z \Q$ is non-split.
\end{enumerate}
If this is the case, then $\cZ(\bL)_{\mathrm{red}}$ is in fact a single point.
\end{corollary}
\begin{proof}
The first statement of the corollary follows from Theorem \ref{thmA} directly. If this is the case, then $\cZ(\bL)_{\mathrm{red}}$ is a single point by Theorem \ref{thmB}.
\end{proof}

We now proceed to prove Theorem \ref{thmA} and \ref{thmB}. Define $L$ as in \eqref{eq:L}. Then we can replace all conditions on $\bL$ in Theorem \ref{thmA} and \ref{thmB} by the same conditions on $L$. Moreover $\mathfrak{d}(L)=\mathfrak{d}(\bL)$.
\subsection{Proof of Theorem \ref{thmA}}\label{subsec:proofofthmA}
It suffices to show the corresponding statements on $\bar\bk$-points, namely, 
\[\cW(L)=\bigcup_{\{\Lambda\in \mathrm{Vert}(L)\mid t(\Lambda)=\mathfrak{d}(L)\}}\cV(\Lambda).\]
For $\bx\in \bV^n$, fix a Jordan splitting of $L$ as in \eqref{Jordansplitting}. We then have
\[L=L_0\obot L_{\geq 1}, \ L^\sharp =L_0\obot (L_{\geq 1})^\sharp.\]
For any $\Lambda \in \mathrm{Vert}(L)$, by Proposition \ref{prop:startinglemma} we have
\begin{equation}\label{eq:LLambdasequence}
    L\subseteq \Lambda^\sharp \subseteq \Lambda \subseteq L^\sharp. 
\end{equation}
If $L_0\neq \{0\}$ then $ s\Lambda=s\Lambda^\sharp =sL= 0$. By Lemma \ref{directsummand} we can assume 
\[\Lambda = L_0\obot \Lambda'.\]
Then $\Lambda^\sharp=L_0 \obot (\Lambda')^\sharp$ and we have the sequence 
\[L_{\geq 1}\subseteq (\Lambda')^\sharp\subseteq \Lambda' \subseteq (L_{\geq 1})^\sharp.\]
As the map $\Lambda\mapsto \Lambda'$ above is a bijection and $\mathfrak{d}(L)=\mathfrak{d}(L_{\geq 1})$, in order to prove Theorem \ref{thmA} we can without loss of generality assume
\begin{equation}\label{eq:L_0=0}
    L_0=0 \text{ or equivalently } \frac{1}{\pi}L\subseteq L^\sharp
\end{equation}
in the rest of the subsection. Define
\[m:=m(L)=\mathrm{rank}_{\Oo_F} (L_{\geq 1}),\]
which is the same as $\mathrm{rank}_{\Oo_F}(L)$ by assumption \eqref{eq:L_0=0}.
In the rest of the section we simply write $\mathrm{rank}(\Lambda)$ instead of $\mathrm{rank}_{\Oo_F}(\Lambda)$ for an $\Oo_F$-lattice $\Lambda$.

The fact that $\mathfrak{d}(L)$ can be no bigger than the bounds stated in Theorem \ref{thmA} is a restatement of \cite[Lemma 3.3]{RTW}. Our goal is to prove that it can achieve that number. To be more precise, we prove that if $\Lambda\in \mathrm{Vert}(L)$ and $t(\Lambda)<\mathfrak{d}(L)$, then there is a $\Lambda' \in \mathrm{Vert}(L)$ such that $ \Lambda \subset \Lambda'$ (hence $ \cV(\Lambda)\subset \cV(\Lambda')$) and $t(\Lambda')=\mathfrak{d}(L)$.

From now on assume $\Lambda\in \mathrm{Vert}(L)$, namely \eqref{eq:LLambdasequence} holds.
Let $t=t(\Lambda)$, then $\pi \Lambda \subset^{m-t} \Lambda^\sharp \subset^t \Lambda $. Define 
\[r:=\mathrm{dim}_{\bk}\left(\left(\frac{1}{\pi} \Lambda^\sharp \cap L^\sharp\right) /\Lambda\right).\]
Since $\Lambda^\sharp /\left(\frac{1}{\pi} \Lambda^\sharp \cap L^\sharp\right)^\sharp=\Lambda^\sharp /(\pi \Lambda+L)$, we have the following chain of inclusions
\begin{equation}\label{eq:keychain}
    \pi \Lambda +L \subset^r \Lambda^\sharp \subset^t \Lambda \subset^r\frac{1}{\pi} \Lambda^\sharp \cap L^\sharp.
\end{equation}
Our assumption $L=L_{\geq 1}$ implies that $L\subseteq \pi L^\sharp$. This together with \eqref{eq:LLambdasequence} and \eqref{eq:keychain} implies that
\begin{equation}\label{pilambda+L}
    \pi \Lambda+L \subseteq \pi \left(\frac{1}{\pi} \Lambda^\sharp \cap L^\sharp\right)=\Lambda^\sharp \cap \pi L^\sharp.
\end{equation}
Hence
\[\mathrm{dim}_{\bk}\left(\left(\frac{1}{\pi} \Lambda^\sharp \cap L^\sharp\right)/(\pi \Lambda+L)\right)\geq \mathrm{dim}_{\bk}\left(\left(\frac{1}{\pi} \Lambda^\sharp \cap L^\sharp\right)/\pi \left(\frac{1}{\pi} \Lambda^\sharp \cap L^\sharp\right)\right)=m.\]
Notice that the first quotient in the above inequality is indeed a $\bk$ vector space. Combine the above inequality with \eqref{eq:keychain}, we have
\begin{equation}\label{2r+t}
    2r+t\geq m.
\end{equation}

Define a $\bk$-valued symmetric form $S(\cdot,\cdot)$ on the $\bk$-vector space $\frac{1}{\pi}\Lambda^\sharp/\Lambda$ by
\[S(x,y):=\delta\pi_0\langle \pi x,y\rangle .\]
\begin{lemma}
Suppose $\Lambda$ is a vertex lattice in $\mathrm{Vert}(L)$ such that $\mathrm{dim}_{\bk}\left(\frac{1}{\pi} \Lambda^\sharp \cap L^\sharp/\Lambda\right) \geq 3$, then there exists a lattice $\Lambda' \in \mathrm{Vert}(L)$ with $\Lambda \subset \Lambda'$ and $t(\Lambda')> t(\Lambda)$.
\end{lemma}
\begin{proof}
Recall that every quadratic form on a $\bk$-vector ($\bk$ is finite) space with dimension bigger or equal to three has an isotropic line by the Chevalley-Warning Theorem. Take an isotropic line $\ell$ in $\frac{1}{\pi} \Lambda^\sharp \cap L^\sharp/\Lambda$. Let $\Lambda'=:\mathrm{pr}^{-1}(\ell)$ where $\mathrm{pr}$ is the natural projection $\frac{1}{\pi} \Lambda^\sharp \rightarrow \frac{1}{\pi}\Lambda^\sharp /\Lambda$. The fact that $\ell$ is isotropic just means
\[\delta\pi_0 \langle \pi \Lambda',\Lambda'\rangle \subseteq \pi_0 \Oo_{F_0}.\]
This shows that $\pi \Lambda' \subseteq (\Lambda')^\sharp$. Since $\Lambda \subset^1 \Lambda'$ we have 
\[(\Lambda')^\sharp\subset^1 \Lambda^\sharp \subseteq\Lambda \subset^1 \Lambda'.\]
So $ \Lambda'$ is a vertex lattice and $t(\Lambda')=t(\Lambda)+2$. Since $\Lambda\subseteq L^\sharp$, by the definition of $\Lambda'$, we also have $\Lambda'\subseteq L^\sharp$.
In other words, $\Lambda'\in \mathrm{Vert}(L)$. The lemma is proved.
\end{proof}

By induction using the above lemma and the fact that $\cV(\Lambda)\subset \cV(\Lambda')$ if $\Lambda \subset \Lambda'$ (\cite[Proposition 4.3]{RTW}), we reduce to the case when $r=\mathrm{dim}_{\bk}(\frac{1}{\pi} \Lambda^\sharp \cap L^\sharp /\Lambda) \leq 2$. Also keep in mind that equation (\ref{2r+t}) holds. There are at most four cases when $r\leq 2$ and $t(\Lambda)$ is smaller than the claimed $\mathfrak{d}(L)$ in Theorem \ref{thmA}:
\begin{enumerate}
    \item $m$ is even, $t(\Lambda)=m-2,r=2$;\
    \item $m$ is even, $t(\Lambda)=m-2,r=1$;\
    \item $m$ is even, $t(\Lambda)=m-4,r=2$;\
    \item $m$ is odd, $t(\Lambda)=m-3,r=2$.
\end{enumerate}
We will show that $\Lambda$ can be enlarged to $\Lambda'$ so that $t(\Lambda')=\mathfrak{d}(L)$ case by case.

\noindent
Case (1): we have
\[\Lambda^\sharp \subset^{m-2} \Lambda \subset^2 \frac{1}{\pi } \Lambda^\sharp \cap L^\sharp \subseteq \frac{1}{\pi } \Lambda^\sharp.\]
Since $\Lambda^\sharp \subset^m \frac{1}{\pi} \Lambda^\sharp$, we actually have $\frac{1}{\pi} \Lambda^\sharp \subseteq L^\sharp$. Choose a Jordan splitting of $\Lambda$
\[\Lambda=\Lambda_0\obot \Lambda_{-1}.\]
Then we know rank$(\Lambda_0)=2$, $\Lambda_0^\sharp =\Lambda_0$, rank$(\Lambda_{-1})=m-2$ and $\pi \Lambda_{-1}=\Lambda_{-1}^\sharp$. By Proposition \ref{prop:modularlattice}, $\Lambda_{-1}\otimes_\Z \Q$ is split. If $\Lambda_0\otimes_\Z \Q$ is split, then there exist $ e_1,e_2\in \Lambda_0$ such that $(e_1,e_1)=(e_2,e_2)=0$, $(e_1,e_2)=1$. Define 
\[\Lambda'=\Lambda_{-1}\obot \mathrm{span}\{e_1,\pi^{-1} e_2\}.\]
By definition $\Lambda'\subset \frac{1}{\pi} \Lambda^\sharp$.
By the fact that $\frac{1}{\pi} \Lambda^\sharp \subseteq L^\sharp$, we know that  $\Lambda'\subseteq L^\sharp$. Also
\[(\Lambda')^\sharp=\Lambda_{-1}^\sharp\obot \mathrm{span}\{\pi e_1, e_2\}.\]
So $t(\Lambda')=m$. Hence $t(\Lambda')=\mathfrak{d}(L)$ as stated in Theorem \ref{thmA}. If $\Lambda_0\otimes_\Z \Q$ is nonsplit, then $t(\Lambda)=m-2$ already obtains the number $\mathfrak{d}(L)$ as stated in Theorem \ref{thmA}.

\noindent
Case (2): we have
\[\pi \Lambda +L \subset^1 \Lambda^\sharp \subset^{m-2} \Lambda \subset^1 \frac{1}{\pi} \Lambda^\sharp \cap L^\sharp,\text{ and } \pi \left(\frac{1}{\pi} \Lambda^\sharp \cap L^\sharp\right) \subset^m \frac{1}{\pi}\Lambda^\sharp \cap L^\sharp.\]
We have already seen in equation (\ref{pilambda+L}) that
\[\pi \Lambda+L \subseteq \pi \left(\frac{1}{\pi} \Lambda^\sharp \cap L^\sharp\right)=\Lambda^\sharp \cap \pi L^\sharp.\]
These together imply that in fact $\pi\Lambda+L=\Lambda^\sharp \cap \pi L^\sharp$. But
\[\pi \Lambda +L=\left(\frac{1}{\pi} \Lambda^\sharp \cap L^\sharp\right)^\sharp.\]
So define $\Lambda'=\frac{1}{\pi} \Lambda^\sharp \cap L^\sharp$, we have $t(\Lambda')=m$. This implies that $\Lambda'\otimes_\Z \Q$ is split and $t(\Lambda')=\mathfrak{d}(L)$. 

\noindent
Case (3):
Similar to case (2).

\noindent
Case (4): we have 
\[\Lambda^\sharp \subset^{m-3} \Lambda \subset^2 \frac{1}{\pi } \Lambda^\sharp \cap L^\sharp \subset^1 \frac{1}{\pi} \Lambda^\sharp.\]
Choose a Jordan splitting of $\Lambda$
\[\Lambda=\Lambda_0\obot \Lambda_{-1}.\]
Then we know rank$(\Lambda_0)=3$, $\Lambda_0^\sharp =\Lambda_0$, rank$(\Lambda_{-1})=m-3$, $\Lambda_{-1}=\frac{1}{\pi}\Lambda_{-1}^\sharp$. By assumption there is a basis $\{e_1,e_2,e_3\} $ of $\Lambda_0$ such that 
\[\frac{1}{\pi}e_1, \frac{1}{\pi} e_2 \in \frac{1}{\pi}\Lambda^\sharp \cap L^\sharp,\frac{1}{\pi}e_3 \notin L^\sharp.\]
By changing $\{e_1,e_2\}$ by an $\Oo_F$ linear combination of them, we can assume $(e_i,e_i)=u_i\ (i=1,2)$ for $u_i\in \Oo_{F_0}^\times$ and $(e_1,e_2)=0$. By modifying $e_3$ using linear combinations of $e_1,e_2$ we can in fact assume that under the basis $\{e_1,e_2,e_3\}$, the form $(\cdot,\cdot)|_{\Lambda_0}$ is represented by the diagonal matrix diag$\{u_1,u_2,u_3\}$ with $u_1,u_2,u_3\in \Oo_{F_0}^\times$. This means that 
\[(e_3,L^\sharp)=(e_3,e_3)\Oo_F=\Oo_F \Rightarrow e_3 \in L.\]
But $\frac{1}{\pi} e_3\notin L^\sharp $, these together contradict our assumption \eqref{eq:L_0=0}. In conclusion, case (4) is not possible under the assumption \eqref{eq:L_0=0}.

This finishes the proof of Theorem \ref{thmA}.
\qedsymbol

\subsection{Proof of Theorem \ref{thmB}}\label{subsec:proofofthmB}
Again it suffices to show the corresponding statements on $\bar\bk$-points.
By Theorem \ref{thmA}, $\cW(L)=\cV(\Lambda)$ is true if and only if $\Lambda$ is the unique lattice in Vert$(L)$ with $t(\Lambda)=\mathfrak{d}(L)$. As in the proof of Theorem \ref{thmA}, we can assume \eqref{eq:L_0=0}.
\begin{lemma}
Assume that one of the following conditions holds,
\begin{enumerate}
    \item $n_{\mathrm{even}}\geq 3$ or $n_{\mathrm{even}}=2$ with $L_{\geq 2}\otimes \Q$ split. \
    \item $n_{\mathrm{odd}}\geq 2$. 
\end{enumerate}
Then there is more than one $\Lambda$ in Vert$(L)$ such that $t(\Lambda)=\mathfrak{d}(L)$.
\end{lemma}
\begin{proof}
Fix a Jordan splitting of  $L$ as in \eqref{Jordansplitting}.
If $n_{\mathrm{odd}}\geq 2$ by Proposition \ref{prop:modularlattice}, we can find a direct summand $H(i), i\geq 2$ of $L$.
If $n_{\mathrm{even}}\geq 3$, scale the sub $\Oo_F$-module $L_{\geq 2}$
to be $\pi^2$-modular to get a new lattice $L'\supseteq L$ such that $L'_{2}$ has rank bigger or equal to $3$. Notice that $\mathfrak{d}(L')=\mathfrak{d}(L)$.
\cite[Proposition 63:19]{OMeara} shows that every quadratic space over a local field with dimension greater or equal to $5$ is isotropic. We apply this to the trace form of $(\cdot,\cdot)|_{L'_2}$ and conclude that there is a maximal element in $L'_2$ that has length zero. Hence there is an $H(i), i\geq 2$ which is a direct summand of $L'_2$. 
Similarly if $n_{\mathrm{even}}=2$ and $L_{\geq 2}\otimes \Q$ is split, we can find a lattice $L'\supseteq L$ such that $\mathfrak{d}(L)=\mathfrak{d}(L')$ and a direct summand $H(i)$ ($i\geq 2$) of $L'$. 
In any case we can find a a lattice $L'\supseteq L$ such that $\mathfrak{d}(L)=\mathfrak{d}(L')$ and 
\[L'=L''\obot H(a),\]
with $a\geq 2$. In particular $\mathrm{Vert}(L')\subseteq \mathrm{Vert}(L)$.

Notice that $\mathfrak{d}(L')=\mathfrak{d}(L'')+2$. By Theorem \ref{thmA}, there is a vertex lattice $\Lambda\in \mathrm{Vert}(L'')$ such that $t(\Lambda)=\mathfrak{d}(L'')$.
Let $\{e_1,e_2\}$ be a basis of $H(a)$ such that $(e_1,e_1)=(e_2,e_2)=0$ and $(e_1,e_2)=\pi^a$. Define
\[\Lambda_1:=\Lambda\obot \mathrm{span}\{\pi^{-a}e_1,\pi^{-1}e_2\},\quad \Lambda_2:=\Lambda\obot \mathrm{span}\{\pi^{-a}e_2,\pi^{-1}e_1\}.\]
Then
\[\Lambda_1^\sharp=\Lambda^\sharp\obot \mathrm{span}\{\pi^{-a+1}e_1,e_2\},\quad \Lambda_2^\sharp=\Lambda^\sharp\obot \mathrm{span}\{\pi^{-a+1}e_2,e_1\}.\]
This shows that $t(\Lambda_1)=t(\Lambda_2)=\mathfrak{d}(L)$ and $\Lambda_1,\Lambda_2\in \mathrm{Vert}(L')$, but $\Lambda_1\neq \Lambda_2$. This proves the lemma.
\end{proof}
This proves the ``only if" part of Theorem \ref{thmB}. To prove the converse, we start with a lemma.

\begin{lemma}\label{lem:L_0L_1inLambda}
Suppose $L=L_0\obot L_1\obot L_{\geq 1}$(Jordan splitting). If $\Lambda$ is a vertex lattice in $\mathrm{Vert}(L)$ such that $t(\Lambda)=\mathfrak{d}(L)$, then $L_0\obot L_1^\sharp \subset \Lambda$.
\end{lemma}
\begin{proof}
Suppose $L_0\obot L_1^\sharp \not\subset \Lambda$. Let $\Lambda':=\Lambda+ L_0\obot L_1^\sharp$. We have $L_0^\sharp=L_0$ and $\pi L_1^\sharp=L_1$. Then $L^\sharp=L_0\obot\frac{1}{\pi} L_1 \obot L_{\geq 1}^\sharp$ and
\[(\Lambda')^\sharp=\Lambda^\sharp \cap(L_0\obot L_1\obot (L_{\geq 1}\otimes \Q)).\]
Using the above equation and the fact that $\Lambda\in \mathrm{Vert}(L)$, one checks immediately that $\pi \Lambda'\subseteq (\Lambda')^\sharp$. Since $\Lambda\subset\Lambda'$ and $ \Lambda^\sharp \subseteq \Lambda$, we have $ (\Lambda')^\sharp \subseteq \Lambda'$.
Also $ \Lambda'\subseteq L^\sharp$, so $ \Lambda'\in \mathrm{Vert}(L)$. But $t(\Lambda')> t(\Lambda)$, which contradicts the maximality of $t(\Lambda)$ among vertex lattices in $\mathrm{Vert}(L)$.
\end{proof}

Now we assume conditions (1) and (2) of Theorem \ref{thmB} hold. By Proposition \ref{prop:modularlattice}, we have the following three cases
\begin{enumerate}
    \item $L\simeq L_0\obot H(1)^\ell$ \
    \item $L\simeq L_0\obot H(1)^\ell\obot(u(-\pi_0)^a)$ with $a\geq 1$ and $ u\in \Oo_{F_0}^\times$.\
     \item $L\simeq L_0\obot H(1)^\ell\obot(u_1(-\pi_0)^a)\obot (u_2 (-\pi_0)^b)$, where $u_1,u_2\in \Oo_{F_0}^\times ,-u_1 u_2 \notin \mathrm{Nm}_{F/F_0}(F/F_0)$ and $a,b$ are integers greater or equal to $1$.
\end{enumerate}
We need to prove that in each case there is a unique $\Lambda\in \mathrm{Vert}(L)$ such that $t(\Lambda)=\mathfrak{d}(L)$.
Cases (1) follows from Lemma \ref{lem:L_0L_1inLambda} directly (in this case $\Lambda$ has to be $L_0\obot H(1)^\ell$). Case (2) follows from Lemma \ref{lem:L_0L_1inLambda} and simple arguments.
Now we prove (3). By Lemma \ref{lem:L_0L_1inLambda} and Lemma \ref{directsummand}, it suffices to prove the statement for $L=(u_1(-\pi_0)^a)\obot (u_2 (-\pi_0)^b)$ .

Let $L=\mathrm{span}\{e_1,e_2\}$ and $T=$diag$\{u_1(-\pi_0)^a,u_2 (-\pi_0)^b\}$ is the gram matrix of $\{e_1,e_2\}$. 
Suppose $\Lambda=\mathrm{span}\{[e_1,e_2]S\}\in\mathrm{Vert}(L)$ where $S \in \mathrm{GL}_2(F)/\mathrm{GL}_2(\Oo_F)$. 
Since $L\otimes_\Z\Q$ is nonsplit, we must have $\Lambda^\sharp=\Lambda$.
Then $\Lambda^\sharp=[e_1,e_2] T^{-1} {}^t\overline{S}^{-1}$ and 
\[\Lambda^\sharp =\Lambda \Leftrightarrow S^{-1}  T^{-1} {}^t\overline{S}^{-1} \in \mathrm{GL}_2(\Oo_F)  \Leftrightarrow {}^t \overline{S} T S \in \mathrm{GL}_2(\Oo_F)\]
\[L\subseteq \Lambda^\sharp \Leftrightarrow {}^t \overline{S}  T\in M_2(\Oo_F).\]
Apply Proposition \ref{prop:modularlattice} and multiply $S$ on the right by an element in $\mathrm{GL}_2(\Oo_F)$ if necessary, we can assume 
\[{}^t\overline{S} T S=\left(\begin{array}{cc}
    u_1 & 0 \\
    0 & u_2
\end{array}\right)=:T_1.\]
Assume 
\[S=\left(\begin{array}{cc}
\pi^{-a}  &  0\\
0 & \pi^{-b}
\end{array}\right)S_0,\]
then ${}^t\overline{S}_0 T_1 S_0=T_1$. Claim: $S_0\in \mathrm{GL}_2(\Oo_F)$. 
Assume $S_0=\left(\begin{array}{cc}
x  &  y\\
z & w
\end{array}\right)$, then ${}^t\overline{S}_0 T_1 S_0=T_1$ implies that
\begin{align*}
    u_1 x\overline{x}+u_2z\overline{z}=&u_1 \\
    u_1 \overline{y}x+u_2\overline{w}z=&0 \\
    u_1 y\overline{y}+u_2 w\overline{w}=&u_2.
\end{align*}
If $z=0$, then $y=0$ and $x,w\in \Oo_F^\times$. If $x=0$, then $w=0$ and $y,z\in \Oo_{F}^\times$ as $u_1,u_2$ are units.

Now assume that $xz\neq0$.
Suppose $x=x_0 \pi^e$ where $e<0,x_0\in \Oo_F^\times$, then
\[x\overline{x}-1=(-\pi_0)^e(x_0\overline{x}_0-(-\pi_0)^{-e}).\]
Since $F$ is ramified over $F_0$, $\mathrm{Nm}_{F/F_0}(\Oo_F^\times/\Oo_{F_0}^\times)=(\Oo_{F_0}^\times)^2$ by class field theory. As $x_0\overline{x}_0 \in \mathrm{Nm}_{F/F_0}(\Oo_F^\times/\Oo_{F_0}^\times)=(\Oo_{F_0}^\times)^2$, by Hensel's lemma, $x_0\overline{x}_0-(-\pi_0)^{-e}\in \mathrm{Nm}_{F/F_0}(\Oo_F^\times/\Oo_{F_0}^\times) $. Then
\[-\frac{u_2}{u_1}=\frac{x\overline{x}-1}{z\overline{z}}=\frac{(-\pi_0)^e(x_0\overline{x}_0-(-\pi_0)^{-e})}{z \bar{z}}\in \mathrm{Nm}_{F/F_0}(\Oo_F^\times/\Oo_{F_0}^\times),\]
contradicts our assumption on $-u_1 u_2$. This show that $e\geq 0$ and $x\in \Oo_F$, then $z\in \Oo_F$ too.

Similarly $u_1 y\overline{y}+u_2 w\overline{w}=u_2$ implies $y,w\in \Oo_F $ if $yw\neq 0$. This proves the claim that $S_0\in \mathrm{GL}_2(\Oo_F)$. In other words
\[\Lambda=\mathrm{span}\{\pi^{-a} e_1,\pi^{-b} e_2\}.\]
This proves the uniqueness of $\Lambda$ and we finish the proof of Theorem \ref{thmB}.
\qedsymbol

\section{Unitary Shimura varieties}\label{sec:shimuravariety}
In this section we briefly review the definition of an integral model of unitary Shimura variety following \cite[Section 6]{RSZshimura} (see also \cite{RSZdiagonal} and \cite{cho2018basic}). Let $F$ be a CM field over $\Q$ with totally real subfield $F_0$ of index $2$ in it. Let $d=[F_0:\Q]$.
We denote by $a\mapsto \bar{a}$ the nontrivial automorphism of $F/F_0$. Define
\begin{equation}\label{eq:Vram}
    \mathscr{V}_{\mathrm{ram}}=\{\text{finite places } v \text{ of }F_0\mid v\text{ ramifies in }F\}.
\end{equation}
In this paper we assume that $\mathscr{V}_{\mathrm{ram}}$ is nonempty. We also make the assumption as in \cite[Section 6]{RSZshimura} that every $v\in \mathscr{V}_{\mathrm{ram}}$ is unramified over $\Q$ and does not divide $2$.

Fix a totally imaginary element $\sqrt{\Delta}\in F$. Denote by $\Phi_{F_0}$ (resp. $\Phi_F$) the set of real (resp. complex) embeddings of $F_0$ (resp. $F$). Define a CM type of $F$ by 
\begin{equation}
    \Phi=\{\varphi\in\Phi_F\mid \varphi(\sqrt{\Delta})\in \sqrt{-1} \R_{>0}\}.
\end{equation}
We fix a distinguished element $\varphi_0\in \Phi$. For $\varphi \in \Hom_\Q(F,\C)$, denote its complex conjugate by $\bar\varphi$.

\subsection{The Shimura datum}
Define a function $r:\Hom_\Q(F,\C)\rightarrow \Z_{\geq 0}$ by 
\[\varphi\mapsto r_\varphi:=\left\{\begin{array}{cc}
        1 &   \text{ if } \varphi=\varphi_0;\\
        0 & \text{ if } \varphi\in \Phi,\varphi\neq\varphi_0; \\
        n-r_{\bar{\varphi}} & \text{ if } \varphi\notin \Phi.
\end{array}\right.\]
Assume that $W$ is a $n$ dimensional $F$-vector space with a Hermitian form $(\cdot,\cdot)$ such that
\[\mathrm{sig}W_\varphi=(r_\varphi,r_{\bar{\varphi}}), \forall \varphi\in \Phi\]
where $W_\varphi:=W\otimes_{F,\varphi} \C$ and $\mathrm{sig}W_\varphi $ is its signature with respect to $(\cdot,\cdot)$. Let $\rU(W)$ (resp. $\mathrm{GU}(W)$) be the unitary group (resp. general unitary group) of $(W,(\cdot,\cdot))$. Recall that for an $F_0$-algebra $R$, we have 
\[\mathrm{GU}(W)(R)=\{g\in \mathrm{GL}(W\otimes_{F_0}R) \mid (gv,gw)=c(g)(v,w),\forall v,w \in W\otimes_{F_0}R\}.\]
Define the following groups.
\[Z^\Q:=\{z\in\mathrm{Res}_{F/\Q} \bG_m\mid \mathrm{Nm}_{F/F_0}(z)\in \bG_m\},\]
\[G=\mathrm{Res}_{F_0/\Q}\rU(W),\]
\begin{equation}\label{eq:G^Q}
   G^\Q:=\{g\in\mathrm{Res}_{F_0/\Q} \mathrm{GU}(W)\mid c(g)\in \bG_m \}. 
\end{equation}
Notice that
\[Z^\Q(\R)=\{(z_{\varphi})\in (\C^\times)^\Phi\mid |z_\varphi|=|z_{\varphi_0}|,\forall \varphi \in \Phi\}.\]
Define the Hodge map 
\[h_{Z^\Q}: \C^\times \rightarrow Z^\Q(\R), z \mapsto (\bar{z},\ldots,\bar{z}).\]
For each $\varphi\in \Phi$ choose a $\C$-basis of $W_\varphi$ such that $(\cdot,\cdot)$ is given by the matrix $\mathrm{diag}(1_{r_\varphi},-1_{r_{\bar\varphi}})$. Define the Hodge map 
\[h_{\mathrm{GU}(W)}:\C^\times \rightarrow \mathrm{Res}_{F_0/\Q}\mathrm{GU}(W)(\R)\cong \prod_{\varphi\in \Phi} \mathrm{GU}(W_\varphi)\]
by sending $z$ to $ \mathrm{diag}(z\cdot 1_{r_\varphi},\bar{z} \cdot 1_{r_{\bar\varphi}})$ for each $\varphi$ component. Then there exists $h_{G^\Q}:\C^\times \rightarrow G^\Q(\R)$ such that $h_{\mathrm{GU}(W)}$ factors as
\[h_{\mathrm{GU}(W)}=i\circ h_{G^\Q}\] 
where $i:G^\Q(\R)\rightarrow \mathrm{Res}_{F_0/\Q}\mathrm{GU}(W)(\R)$ is the natural inclusion.

Define 
\[\tilde{G}:=Z^\Q\times_{\bG_m} G^\Q\]
where the maps from the factors on the right hand side to $\bG_m$ are $\mathrm{Nm}_{F/F_0}$ and the similitude character $c(g)$ respectively. Notice that the map
\begin{equation}\label{tildeG=ZG}
    \tilde{G} \rightarrow Z^\Q\times G, (z,g)\mapsto (z,z^{-1} g)
\end{equation}
is an isomorphism. We define the Hodge map $h_{\tilde G}$ by 
\[h_{\tilde G}:  \C^\times \rightarrow \tilde{G}(\R), z\mapsto (h_{Z^\Q}(z),h_{G^\Q}(z)).\]
Then $(\tilde{G},h_{\tilde G})$ is a Shimura datum whose reflex field $E\subset \bar{\Q}$ is defined by 
\begin{equation}
    \mathrm{Aut}(\bar{\Q}/E)=\{\sigma \in \mathrm{Aut}(\bar{\Q})\mid \sigma \circ \Phi=\Phi,\sigma^*(r)=r\}.
\end{equation}
\begin{remark}\label{rmk:reflexfield}
$F$ always embeds into $E$ via $\varphi_0$ (\cite[Remark 3.1]{RSZdiagonal}). Furthermore $E=F$ when $F$ is Galois over $\Q$ or when $F=F_0 K$ where $K$ is an imaginary quadratic field over $\Q$ and $\Phi$ is induced from a CM type of $K/\Q$. From now on we identify $F_0$ as a subfield of $E$ via $\varphi_0$.
\end{remark}
For a small enough compact group $K\in \tilde{G}(\bA_f)$, we can define a Shimura variety $S(\tilde{G},h_{\tilde G})_K$ which has a canonical model over the $\Spec E$. We refer to \cite[Section 3]{RSZshimura} for the moduli problem $S(\tilde{G},h_{\tilde G})_K$ represents. 

\subsection{Integral Model}
In this subsection, we define the integral model for $S(\tilde{G},h_{\tilde G})_K$ (as a Deligne-Mumford stack) in terms of a moduli functor for a particular choice of $W$ and $K$. We remark here that all the results in this section is semi-global in natural so we could instead describe our results on semi-global integral models defined as in \cite[Section 4]{RSZshimura} which will allow a wider choices of $W$ and $K$. It takes only slight modifications to adjust our results to the semi-global setting so we leave it to the interested readers. 

For a lattice $\Lambda$ in $W$, we let $\Lambda^\vee$ denote its dual with respect to the symplectic form $\mathrm{tr}_{F/\Q}(\sqrt{\Delta}^{-1}(\cdot,\cdot))$ and $\Lambda^\sharp$ denote its dual with respect to the Hermitian form $(\cdot,\cdot)$. Then we have 
\begin{equation}\label{eq:Lambdadual}
    \Lambda^\vee=\sqrt{\Delta}\partial^{-1} \Lambda^\sharp
\end{equation}
where $\partial$ is the different ideal of $F/\Q$.
From now on we assume that $W$ contains a lattice $\Lambda$ such that $\Lambda^\vee=\Lambda$. Define the compact subgroup $K_{G}\subset G(\bA_f)$ by
\begin{equation}\label{eq:K_G}
  K_G:=\{g\in G(\bA_f)\mid g(\Lambda\otimes \hat{\Z})=\Lambda\otimes \hat{\Z}\}.  
\end{equation}
Also let $K_{Z^\Q}$ be the unique maximal compact subgroup of $Z^\Q(\bA_f)$:
\begin{equation}\label{eq:K_Z}
  K_{Z^\Q}:=\{z\in (\Oo_F\otimes \hat{\Z})^\times\mid \mathrm{Nm}_{F/F_0}(z)\in \hat{\Z}\}.  
\end{equation}
Define the compact subgroup
\begin{equation}\label{eq:K}
    K:=K_{Z^\Q}\times K_G \subset \tilde{G}(\bA_f)
\end{equation}
under the isomorphism \eqref{tildeG=ZG}.

First we define an auxilary moduli functor $\cM_0$ over $\Spec \Oo_E$.  For a locally notherian $\Oo_E$-scheme $S$, we define $\cM_0 (S)$ to be the groupoid of triples $(A_0,\iota_0,\lambda_0)$ where 
\begin{enumerate}
    \item $A_0$ is an abelian scheme over $S$.
    \item $\iota_0: \Oo_F\rightarrow\End (A_0)$ is an $\Oo_F$-action satisfying the Kottwitz condition of signature $((0,1)_{\varphi\in \Phi})$, namely
    \[\mathrm{charpol}(\iota_0(a)\mid \Lie A_0)=\prod_{\varphi\in \Phi}(T-\bar{\varphi}(a)),\forall a\in \Oo_F.\]
    \item $\lambda_0$ is a principal polarization of $A_0$ whose Rosati involution induces on $\Oo_F$ via $\iota_0$ the nontrivial Galois automorphism of $F/F_0$.
\end{enumerate}
A morphism between two objects $(A_0,\iota_0,\lambda_0)$ and $(A'_0,\iota'_0,\lambda'_0)$ is a $\Oo_F$-linear isomorphism $A_0\rightarrow A'_0$ that pulls $\lambda'_0$ back to $\lambda$.

Since we assume $\mathscr{V}_{\mathrm{ram}}$ is nonempty, $\cM_0$ is nonempty (\cite[Remark 3.7]{RSZshimura}). Then $\cM_0$ is a Deligne-Mumford stack, finite and \'etale over $\Spec \Oo_E$ (\cite[Proposition 3.1.2]{howardCM}).
Moreover, we choose a $1$ dimensional $F$ vector space $W_0$ such that $W_0$ has an $\Oo_F$ lattice $\Lambda_0$ with a nondegenerate alternating form $\langle\cdot,\cdot\rangle_0$ satisfying 
\begin{enumerate}
    \item $\langle ax,y\rangle_0=\langle x,\bar{a} y\rangle_0$ for all $a\in \Oo_F$ and $x,y\in \Lambda_0$.
    \item The quadratic form $x\mapsto \langle \sqrt{\Delta}x,x\rangle_0$ is negative definite.
    \item The dual lattice $\Lambda_0^\vee$ of $\Lambda_0$ with respect to $\langle\cdot,\cdot\rangle_0$ is $\Lambda_0$.
\end{enumerate}
Let $(\cdot,\cdot)_0$ be the unique Hermitian form on $W_0$ such that $\mathrm{tr}_{F/\Q}(\sqrt{\Delta}^{-1}(\cdot,\cdot)_0)=\langle\cdot,\cdot\rangle_0$.
Then $(W_0,(\cdot,\cdot)_0)$ determines a certain similarity class $\xi$ of hermtian forms which in turn give us an open and closed substack $\cM_0^\xi$ of $\cM_0$ (see \cite[Lemma 3.4]{RSZdiagonal}).
Define the $F$-vector space 
\begin{equation}\label{eq:V}
  V=\Hom_F(W_0,W)  
\end{equation}
with the Hermitian form $(\cdot,\cdot)_V$ determined by $(\cdot,\cdot)$ and $(\cdot,\cdot)_0$ via 
\begin{equation}\label{eq:(,)_V}
    (x(a),y(b))=(x,y)_V (a,b)_0,\forall x,y\in V, \forall a,b\in W_0.
\end{equation}
The lattice 
\begin{equation}\label{eq:globalL}
    L:=\Hom_{\Oo_F}(\Lambda_0,\Lambda)\subset V
\end{equation}
is a self dual lattice with respect to the Hermitian form $(\cdot,\cdot)_V$.

We define the functor $\cM$ on the category of locally notherian schemes over $\Spec \Oo_E$ as follows. For a scheme $S$ in this category, $\cM(S)$ is the groupoid of tuples $(A_0,\iota_0,\lambda_0,A,\iota,\lambda)$ where
\begin{itemize}
    \item $(A_0,\iota_0,\lambda_0)$ is an object of $ \cM_0^{\xi}(S)$.
    \item $A$ is an abelian scheme over $S$.
    \item $\iota:\Oo_F\rightarrow \End(A)$ is an $\Oo_F$-action satisfying the Kottwitz condition of signature $((1,n-1)_{\{\varphi_0\}},(0,n)_{\Phi\backslash \{\varphi_0\}})$, i.e., for all $a\in \Oo_F$
    \[\mathrm{charpol}(\iota(a)\mid \Lie A)=(T-\varphi_0(a))(T-\bar\varphi_0(a))^{n-1}\prod_{\varphi\in \Phi\backslash \{\varphi_0\}}(T-\bar\varphi(a))^n.\]
    \item $\lambda:A\rightarrow A^\vee$ is a principal polarization whose associated Rosati involution induces on $\Oo_F$ via $\iota$ the nontrivial Galois automorphism of $F/F_0$.
\end{itemize}
We assume further that the tuple $(A_0,\iota_0,\lambda_0,A,\iota,\lambda)$
satisfies the sign condition, the Wedge condition and the Eisenstein condition, all of which are defined with respect to the signature $((1,n-1)_{\{\varphi_0\}},(0,n)_{\Phi\backslash \{\varphi_0\}})$.

\begin{enumerate}[label=(H\arabic*)]
\item\label{item:H1}
The sign condition. Let $s$ be a geometric point of $S$ and $(A_{0,s},\iota_{0,s},\lambda_{0,s},A_s,\iota_s,\lambda_s)$ be the pull back of $(A_0,\iota_0,\lambda_0,A,\iota,\lambda)\in \cM(S)$ to $s$. For every nonsplit place $v$ of $F_0$, we impose 
\begin{equation}\label{eq:signcondition}
    \mathrm{inv}_v^r(A_{0,s},\iota_{0,s},\lambda_{0,s},A_s,\iota_s,\lambda_s)=
    \mathrm{inv}_v(V).
\end{equation}
We need to explain the two factors. We refer to \cite[Appedix A]{RSZdiagonal} for the definition of $\mathrm{inv}_v^r(A_{0,s},\iota_{0,s},\lambda_{0,s},A_s,\iota_s,\lambda_s)$. For $\mathrm{inv}_v(V)$, it is defined by 
\[\mathrm{inv}_v(V)=(-1)^{n(n-1)/2}\mathrm{det}(V_v)\in F_{0,v}^\times/ \mathrm{Nm}_{F_v/F_{0,v}} F_v^\times,\]
where $\mathrm{det}(V_v) $ is the determinant of the Hermitian space $V_v:=V\otimes_{F_0} F_{0,v}$. We call this the invariant of $V$ at $v$.
We remark that when $s$ has characteristic zero, the sign condition is equivalent to the condition that there is an isometry 
\begin{equation}\label{eq:signconditionoriginal}
    \Hom_{\bA_{F,f}}(\hat{V}(A_{0,s}),\hat{V}({A_s}))\cong V\otimes_F \bA_{F,f}
\end{equation}
as Hermitian $\bA_{F,f}$-vector spaces. Here $\hat{V}(A_s)$ (resp. $\hat{V}(A_{0,s})$) is the rational Tate module of $A$ (resp. $A_0$). The space $\Hom_{\bA_{F,f}}(\hat{V}(A_{0,s}),\hat{V}({A_s}))$ is equipped with the Hermitian form (\cite[Section 2.3]{KR2})
\begin{equation}\label{eq:hermitianformonTate}
   h(x,y)=\lambda_0^{-1}\circ y^\vee\circ \lambda \circ x\in \End_{\bA_{F,f}}(\hat{V}(A_{0,s})) \cong\bA_{F,f} 
\end{equation}
where $y^\vee$ is the dual of $y$ with respect to the Weil pairings on $\hat{V}(A_{0,s})\times \hat{V}(A^\vee_{0,s})$ and $\hat{V}(A_s)\times \hat{V}(A^\vee_s)$. Hence the sign condition can be seen as a generalization of \eqref{eq:signconditionoriginal}. See \cite[Remark 6.9]{RSZshimura} for cases when the sign condition can be simplified.
\end{enumerate}
The wedge condition and Eisenstein condition are only needed when $S$ has non-empty special fibers in certain characteristics. We temporarily fix a finite prime $p$ of $\Q$. Fix an embedding $\tilde{\nu}:\bar{\Q}\rightarrow \bar{\Q}_p$. This determines a $p$-adic place $\nu$ of $E$. $\tilde\nu$ induces an identification 
\[\Hom_\Q(F,\bar{\Q})\xrightarrow{\sim} \Hom_{\Q}(F,\bar{\Q}_p): \varphi\mapsto \tilde{\nu}\circ \varphi.\]
Let $\mathscr{V}_p(F)$ be the set of places of $F$ over $p$. For each $w\in \mathscr{V}_p(F)$, define
\begin{equation}
    \Hom_w(F,\bar\Q):=\{\varphi \in \Hom_\Q(F,\bar\Q) \mid \tilde{\nu}\circ \varphi \text{ induces }w \}.
\end{equation}
Let $F^t_w$ be the maximal unramified extension of $\Q_p$ in $F_w$. For $\psi\in \Hom_{\Q_p}(F^t_w,\bar{\Q}_p)$, define
\begin{equation}
    \Hom_{w,\psi}(F,\bar\Q):=\{\varphi\in \Hom_{w}(F,\bar{\Q})\mid \tilde{\nu}\circ\varphi|_{F^t_w}=\psi\}.
\end{equation}
The definitions of $\Hom_w(F,\bar\Q)$ and $\Hom_{w,\psi}(F,\bar\Q)$ depend on the choice of $\tilde\nu$ in general but the partition of $\Hom_\Q(F,\bar\Q)$ into unions of $\Hom_{w,\psi}(F,\bar\Q)$ does not (\cite[equation (5.4)]{RSZshimura}).

We make a base change and assume that $S$ is a scheme over $\Spec\Oo_{E,\nu}$ where $\Oo_{E,\nu}$ is the completion of $\Oo_E$ with respect to the $\nu$-adic topology. Then the $\Oo_F$ action on $A$ induces an action of
\[\Oo_F\otimes_\Z \Z_p\cong \prod_{w\in \mathscr{V}_p(F)}\Oo_{F,w}\]
on $\Lie A$. Hence we have a decomposition 
\begin{equation}\label{eq:Liew}
    \Lie A=\bigoplus_{w\in \mathscr{V}_p(F)} \mathrm{Lie}_w A.
\end{equation}
For each $w$, the $\Oo_{F^t_w}$-action on $\mathrm{Lie}_w A$ induces a decomposition 
\begin{equation}\label{eq:Liewpsi}
    \mathrm{Lie}_w A=\bigoplus_{\psi\in \Hom_{\Q_p}(F^t_w,\bar{\Q}_p)} \mathrm{Lie}_{w,\psi} A.
\end{equation}
Here we make a further base change to $\Spec \Oo_{\breve{E}_\nu}$ where $\breve{E}_\nu$ is the completion of the maximal unramified extension of $E_\nu$ in $\bar{\Q}_p$.
\begin{enumerate}[label=(H\arabic*)]\setcounter{enumi}{1}
\item\label{item:H2} The wedge condition. Assume that $w$ is a finite place of $F$ that is ramified over $F_0$. We further assume that the underlying place of $w$ in $\Q$ is $p$ and we make a base change so that $S$ is a $\Spec \Z_p$-scheme. The wedge condition is only needed when $S$ has nonempty special fiber over $\Spec\F_p$.
By our assumption, the underlying place $v$ of $F_0$ is unramified over $\Q$. Hence $F^t_w=F_{0,v}$ and $ \Hom_{w,\psi}(F,\bar{\Q})=\{\varphi_\psi,\overline{\varphi}_\psi\}$ for all $\psi\in \Hom_{\Q_p}(F^t_w,\overline{\Q}_p)$.
For every $\psi$ such that $r_{\varphi_\psi}\neq r_{\overline{\varphi}_\psi}$, decompose $\Lie A$ as in \eqref{eq:Liew} and \eqref{eq:Liewpsi} and impose the wedge condition of \cite{P} (compare with Definition \ref{def:signatureconditions}):
\begin{equation}
    \bigwedge^{r_{\overline{\varphi}_\psi}+1}(\iota(a)-\varphi_\psi(a)\mid \mathrm{Lie}_{w,\psi} A)=0, \ \bigwedge^{r_{{\varphi}_\psi}+1}(\iota(a)-\overline{\varphi}_\psi(a)\mid \mathrm{Lie}_{w,\psi} A)=0
\end{equation}
for all $a\in \Oo_F$. Here since $r_{\varphi_\psi}\neq r_{\overline{\varphi}_\psi}$, $\varphi_\psi$ maps $F_w$ into $E_\nu$, so we can view $\varphi_\psi(a)$ and $\overline{\varphi}_\psi(a)$ as sections in the structure sheaf of the base scheme $S$. 
\item\label{item:H3} The Eisenstein condition. Assume that $w$ is a finite place of $F$ whose underlying place $v$ in $F_0$ is ramified over $\Q$. By our assumption $w$ is unramified over $v$. Again assume that the underlying place of $w$ in $\Q$ is $p$ and we make a base change so that $S$ is a $\Spec \Z_p$-scheme.
Decompose $\Lie A$ as in \eqref{eq:Liew} and \eqref{eq:Liewpsi}.
The Eisenstein condition is a set of conditions on $\mathrm{Lie}_{w,\psi}A$ and is only needed when $S$ has nonempty special fiber over $\Spec\F_p$. We do not describe the condition in detail but instead refer to \cite[Section 5.2, case (1) and (2)]{RSZshimura}.
\end{enumerate}
Finally a morphism between two objects $(A_0,\iota_0,\lambda_0,A,\iota,\lambda)$ and $(A'_0,\iota'_0,\lambda'_0,A',\iota',\lambda')$ is a morphism $(A_0,\iota_0,\lambda_0)\rightarrow (A'_0,\iota'_0,\lambda'_0)$ in $\cM_0^{\xi}(S)$ together with an $\Oo_F$-linear isomorphism $(A,\iota,\lambda)\rightarrow (A',\iota',\lambda')$ that pulls $\lambda'$ back to $\lambda$.

The following Proposition is a partial summarize of \cite[Theorem 3.5, 4.4 and 6.7]{RSZshimura}.
\begin{proposition}
$\cM$ is a Deligne-Mumford stack flat over $\Oo_E$. And 
\[\cM\times_{\Spec \Oo_E} \C=S(\tilde{G},h_{\tilde{G}})_K.\]
Moreover we have:

\noindent(i) $\cM$ is smooth of relative dimension $n-1$ over the open subscheme of $\Spec \Oo_E$ obtained by removing the set $\mathscr{V}_{\mathrm{ram}}(E)$ of finite places $\nu$ of $E$ over $\mathscr{V}_{\mathrm{ram}}$ (see \eqref{eq:Vram}). If $n=1$, then $\cM$ is finite \'etale over all of $\Spec \Oo_E$.

\noindent(ii) If $n\geq 2$, then the fiber of $\cM$ over a place $\nu\in \mathscr{V}_{\mathrm{ram}}(E)$ has only isolated singularities. If $n\geq 3$, then blowing up these isolated points for all $\nu\in\mathscr{V}_{\mathrm{ram}}(E)$ yields a model $\cM^\sharp$ which has semi-stable reduction, hence is regular, over the open subscheme of $\Spec\Oo_E$ obtained by removing all places $\nu\in\mathscr{V}_{\mathrm{ram}}(E)$ that are ramified over $F$. This model $\cM^\sharp$ has a moduli interpretation by \cite{Kr}.
\end{proposition}

\section{Special cycles on the basic locus of unitary Shimura varieties}\label{sec:globalspecialcycle}
\subsection{Definition of the special cycles}
Let $\Herm_m(\Oo_F)$ be the set of $m\times m$ Hermitian matrices with values in $\Oo_F$.  
Let $\Herm_m(\Oo_F)_{\geq 0}$ (resp. $\Herm_m(\Oo_F)_{>0}$) be the subset of totally (i.e. for all archimedean places) positive semidefinite (resp. definite) matrices of $\Herm_m(\Oo_F)$.
We define special cycles as in \cite{KR2} and \cite{RSZshimura}. For a locally notherian scheme $S$ over $\Spec \Oo_E$ and $(A_0,\iota_0,\lambda_0,A,\iota,\lambda)\in \cM(S)$, we have the finite rank locally free $\Oo_F$-module
\[\mathbb{L}(A_0,A):=\Hom_{\Oo_F}(A_0,A).\]
We can define a Hermitian form $h'$ on $\mathbb{L}(A_0,A)$ by assigning for any $x,y\in \mathbb{L}(A_0,A)$
\begin{equation}\label{eq:globalhermitianform}
    h'(x,y)=\iota_0^{-1}(\lambda_0^{-1}\circ y^\vee \circ \lambda \circ x)\in \Oo_F.
\end{equation}
\begin{remark}
The local analogue of $h'(\cdot,\cdot)$ is denoted by $h(\cdot,\cdot)$, see Section \ref{subsec:localspecialcycle}. We use the notation $h'(\cdot,\cdot)$ here to be consistent with \cite{KR2}. 
\end{remark}

\begin{definition}\label{def:global Z T}
Let $T\in \Herm_m(\Oo_F)_{\geq 0}$. The special cycle $\cZ(T)$ is the stack such that for any $\Oo_E$-scheme $S$, $\cZ(T)(S)$ is the groupoid of tuples $(A_0,\iota_0,\lambda_0,A,\iota,\lambda,\bx)$ where $(A_0,\iota_0,\lambda_0,A,\iota,\lambda)\in \cM(S)$ and $\bx=(x_1,\ldots,x_m)\in \mathbb{L}(A_0,A)^m$ such that 
\[h'(\bx,\bx)=(h'(x_i,x_j))=T.\]
\end{definition}
\cite[Proposition 2.9]{KR2} generalizes to our case and shows that tha natural map $\cZ(T)\rightarrow \cM$ is finite and unramified. 

\subsection{Support of the special cycles}
Let $\nu$ be a finite place of $E$ with residue field $k_\nu$ of characteristic $p$. Then $\nu$ determines places $w_0$ of $F$ and $v_0$ of $F_0$ respectively.
For $(A_0,\iota_0,\lambda_0,A,\iota,\lambda)\in \cM(\bar{k}_\nu)$, the $\Oo_F\otimes_\Z \Z_p$-action induces a decomposition of the $p$-divisible group $A[p^\infty]$ and its Dieudonn\'e module
\begin{equation}\label{eq:decompositionofpdivisiblegroupbyw}
   A[p^\infty]=\bigoplus_{w|p} A[w^\infty],\ M(A[p^\infty])=\bigoplus_{w|p} M_{w}(A)
\end{equation}
where $w$ runs over the set of places of $F$ over $p$ and $M_w(A)=M(A[w^\infty])$ for each $w$. 
Each $A[w^\infty]$ admits an $\Oo_{F,w}$ action. 
We say that $(A_0,\iota_0,\lambda_0,A,\iota,\lambda)$ is in the basic locus $\cM_{\nu}^{ss}$ if each $A[w^\infty]$ is isoclinic, i.e., the rational Deudonn\'e module $M_w(A)$ has constant slope for all $w$.

We assume from now on that $T\in \Herm_n(\Oo_F)_{>0}$. 
Then we have the following generalization of \cite[Lemma 2.21]{KR2}.
\begin{lemma}\label{lem:supportofZ}
Assume that $T\in \Herm_n(\Oo_F)_{>0}$. Then $\cZ(T)$ is supported on
\[\bigcup_{\nu} \cM_{\nu}^{ss} \]
where $\nu$ runs over the set of finite places of $E$ whose underlying place of $F_0$ does not split in $F$. 
\end{lemma}
\begin{proof}
The proof is the same as that of \cite[Lemma 8.7]{RSZdiagonal} which is a variant of the proof of \cite[Lemma 2.21]{KR2}.
\end{proof}
For $T\in \Herm_n(\Oo_F)_{>0}$, let $V_T$ be the Hermitian $F$-vector space with gram matrix $T$. Recall that we define a Hermitian vector space $V$ as in \eqref{eq:V}. Define $\mathrm{Diff}(T,V)$ as in \eqref{eq:DiffTVintro} or equivalently
\begin{equation}\label{eq:DiffTV}
    \mathrm{Diff}(T,V):=\{v\text{ is a finite place of }F_0\mid \mathrm{inv}_v(V) \neq \mathrm{inv}_v(V_T)\}.
\end{equation}
Any $v$ in $\mathrm{Diff}(T,V)$ is automatically nonsplit in $F$.
Since $T$ is totally positive definite and $V$ has signature $((n-1,1)_{\{\varphi_0\}},(n,0)_{\Phi\backslash \{\varphi_0\}})$, by Hasse principal, $\mathrm{Diff}(T,V)$ is a finite set of odd cardinality.
The following result generalizes \cite[Proposition 2.22]{KR2}. It should be well-known to experts (c.f.\cite[Section 14.4]{LZ}).
\begin{proposition}\label{prop:precisesupport}
Assume $T\in \Herm_n(\Oo_F)_{>0}$.
\begin{enumerate}
    \item If $|\mathrm{Diff}(T,V)|=\{v_0\}$ where $v_0$ is a finite place of $F_0$, then $\cZ(T)$ is supported on 
\[\bigcup_{\nu\in \mathscr{V}(v_0)} \cM_{\nu}^{ss}\]
where $\mathscr{V}(v_0)$ is the set of places of $E$ over $v_0$.
   \item  $\cZ(T)$ is empty if $|\mathrm{Diff}(T,V)|>1$.
\end{enumerate}
\end{proposition}
\begin{proof}
We prove (1) first.
By Lemma \ref{lem:supportofZ}, we know that $\cZ(T)$ is supported on the basic locus over finite places of $E$. Let $\nu$ be a finite place of $E$ with residue field $k_\nu$ of characteristic $p$ such that $\cZ(T)(\bar{k}_\nu)$ is nonempty. Then $\nu$ determines a place $v_0$ of $F_0$ which does not split in $F$.
Let $(A_0,\iota_0,\lambda_0,A,\iota,\lambda)\in \cZ(T)(\bar{k}_\nu)$. By definition $V_T$ carries the Hermitian form $h'(\cdot,\cdot)$ in \eqref{eq:globalhermitianform}.

When $v$ does not divide $p$, by its definition  $\mathrm{inv}_v^r(A_0,\iota_0,\lambda_0,A,\iota,\lambda)$ is the invariant at $v$ of the Hermitian form $h(\cdot,\cdot)$ defined in \eqref{eq:hermitianformonTate} and is the same as $\mathrm{inv}_v(V)$ by the sign condition. On the other hand,  the invariant at $v$ of the Hermitian form $h(\cdot,\cdot)$ is the same as
$\mathrm{inv}_v(V_T)$ by \cite[Lemma 2.10]{KR2}. 

Now assume $v|p$ and is nonsplit and $w$ is the place of $F$ above $v$. Since the component containing $(A_0,\iota_0,\lambda_0,A,\iota,\lambda)$ has nonempty generic fiber (this is implied for example by \eqref{eq:uniformization} below), \cite[Proposition A1]{RSZdiagonal} tells us that 
\begin{equation}\label{eq:twoinvarethesame}
   \mathrm{inv}_v^r(A_0,\iota_0,\lambda_0,A,\iota,\lambda)=\mathrm{inv}_v(V). 
\end{equation}
On the other hand by \cite[Equation (A.8)]{RSZdiagonal}, we know that  
\[\mathrm{inv}_v^r(A_0,\iota_0,\lambda_0,A,\iota,\lambda)=\mathrm{sgn}(r_{\nu,v})\mathrm{inv}_v(A_0,\iota_0,\lambda_0,A,\iota,\lambda)\]
where by \cite[Equation (A.7)]{RSZdiagonal}
\[\mathrm{sgn}(r_{\nu,v})=\left\{\begin{array}{cc}
    1 & \text{ if } v|p \text{ and }  v\neq v_0, \\
    -1 & \text{ if } v=v_0,
\end{array}\right.\]
and $\mathrm{inv}_v(A_0,\iota_0,\lambda_0,A,\iota,\lambda)$ is the invariant of the Hermitian form on the Dieudonn\'e module (see \eqref{eq:decompositionofpdivisiblegroupbyw}) 
\[\Hom_{F_{w}\otimes_{\Z_p} W(\bar{k}_\nu)}(M_w(A_0)\otimes\Q,M_w(A)\otimes\Q).\]
By Lemma \ref{lem:equivalenceofHermitianforms}, Proposition \ref{prop:comparisonofspecialcycles} and their analogues at inert primes, we know that
\[\mathrm{inv}_v(A_0,\iota_0,\lambda_0,A,\iota,\lambda)=\mathrm{inv}_v(V_T).\]
Hence we have
\[\mathrm{inv}_v(V)=\left\{\begin{array}{cc}
    \mathrm{inv}_v(V_T) & \text{ if } v|p \text{ and }  v\neq v_0, \\
    -\mathrm{inv}_v(V_T) & \text{ if } v=v_0.
\end{array}\right.\]
In conclusion, if $\cZ(T)(\bar{k}_\nu)$ is nonempty we must have 
\[\mathrm{Diff}(T,V)=\{v_0\}.\]
This finishes the proof of the (1).

If $z$ is a geometric point of characteristic zero in $\cZ(T)$, then \eqref{eq:signconditionoriginal} implies that $\mathrm{Diff}(T,V)=\emptyset$. But this is impossible by the signature assumption on $V$ and $V_T$. Hence $\cZ(T)$ has no geometric point of characteristic zero and (2) follows from (1).
\end{proof}

\subsection{Uniformization of the basic locus and special cycles}
We now fix a place $\nu$ of $E$ over $w_0$ of $F$ and $v_0\in \mathscr{V}_{\mathrm{ram}}$ of $F_0$. Let $\breve{E}_\nu$ be the completion of the maximal unramified extension of $E_\nu$. We also denote by $\widehat{\cM}_{\nu}^{ss}$ the completion of $\cM\times_{\Spec \Oo_{E}}\Spec \Oo_{\breve{E}_\nu}$ along its basic locus $\cM_{\nu}^{ss}\times_{\Spec k_\nu}\Spec \bar{k}_\nu$. 

\begin{lemma}\label{lem:nonemptysupersingularlocus}
$\cM_\nu^{ss}(\bar{k}_\nu)$ is nonempty.
\end{lemma}
\begin{proof}
The proof is a variant of that of \cite[Lemma 5.1]{KR2}. 
Let $(A_0,\iota_0,\lambda_0)\in \cM_{0}^{\xi}(\Oo_{\breve{E}_\nu})$. Also let $(A_1,\iota_1,\lambda_1)$ be defined similarly as $ (A_0,\iota_0,\lambda_0)$ except that we change the signature from $(0,1)_{\Phi}$ to $((1,0)_{\{\varphi_0\}},(0,1)_{\Phi\backslash\{\varphi_0\}})$. Both abelian schemes have good reduction at $\nu$ by the smoothness of $\cM_0$. Define 
\[(A,\iota,\lambda'):=(A_0,\iota_0,\lambda_0)^{n-1}\oplus(A_1,\iota_1,\lambda_1).\]
Then $(A_0,\iota_0,\lambda_0,A,\iota,\lambda')\in \cM^{V'}(\Oo_{\breve{E}_\nu})$ where $\cM^{V'}$ has the same definition as $\cM$ except that in the sign condition we replace $V$ by some Hermitian space $V'$ over $F$ with the same signature as $V$. 

From now on we base change to $\Spec \bar{k}_\nu$ and for simplicity denote the base change of $(A_0,\iota_0,\lambda_0,A,\iota,\lambda')$ by the same notation. Then  $(A_0,\iota_0,\lambda_0,A,\iota,\lambda')\in \cM^{V',ss}_\nu$.
Define
\begin{equation}\label{eq:adjustlambda}
   \lambda:=\lambda'\circ (\iota(a/b),1,\ldots,1) 
\end{equation}
where $a,b\in F_0$ represent $\mathrm{det}(V)$ and $\mathrm{det}(V')$ respectively.
Since $V$ and $V'$ have the same the signature over the archimdedean places, $\frac{a}{b}$ is totally positive, hence $\lambda$ is a quasi-polarization. 
Notice that the Rosati involution induced by $\lambda$ on $F\hookrightarrow\End^0(A)$ is the complex conjugation. By the definition of $\lambda$ and the fact that $(A_0,\iota_0,\lambda_0,A,\iota,\lambda')$ satisfies the sign condition \eqref{eq:signcondition} for $V'$ we know that $(A_0,\iota_0,\lambda_0,A,\iota,\lambda)$ satisfies the sign condition for $V$.

By the $\Oo_{F_0}$-action on $A$, we can decompose the $p$-divisible group $A[p^\infty]$ and the rational Dieudonn\'e module $M(A[p^\infty])$ of $A[p^\infty]$ into 
\begin{equation}\label{eq:decompositionpdivisible}
    A[p^\infty]=\bigoplus_{v|p} A[v^\infty],\ M(A[p^\infty])=\bigoplus_{v|p} M_{v}(A)
\end{equation}
where $v$ runs over places of $F_0$ over $p$ and $M_v(A)=M(A[v^\infty])$ for each $v$. Let $M_{v_0}^{\mathrm{rel}}(A)$ be the relative Dieudonn\'e module of $\mathcal{C}_{\bar{k}_\nu}(A[v_0^\infty])$ where $\mathcal{C}$ is the functor in Theorem \ref{thm:comparison}. Choose an $\breve{\Oo}_{F,w_0}$-lattice $\Lambda_{v_0}^{\mathrm{rel}}\subset M_{v_0}^{\mathrm{rel}}(A)\otimes_\Z \Q$ satisfying the condition in Proposition \ref{cV}. Such choice always exists by the non-emptiness statement in Theorem \ref{thm:incidencerelation} (iii). By Theorem \ref{thm:comparison}, $\Lambda_{v_0}^{\mathrm{rel}}$ determines a self dual lattice $\Lambda_{v_0}\in M_{v_0}(A)\otimes_\Z \Q$. For any other $v\neq v_0$ dividing $p$, we can choose a self dual lattice $\Lambda_v\subset M_{v}(A)\otimes_\Z \Q$ with respect to the symplectic form induced by $\lambda$ in a similar manner. Choose a self dual lattice $\Lambda^p\subset \hat{V}^p(A)$ with respect to the symplectic form on $\hat{V}^p(A)$ induced by $\lambda$. These lattices determines an abelian variety $(B,\iota_B,\lambda_B)$ isogenic to $(A,\iota,\lambda)$ where $\lambda_B$ is a principal polarization. Then $(A_0,\iota_0,\lambda_0,B,\iota_B,\lambda_B)\in \cM_{\nu}^{ss}(\bar{k}_\nu)$.
\end{proof}

By the lemma we can choose a framing object $(A_0^o,\iota_0^o,\lambda_0^o,A^o,\iota^o,\lambda^o) \in \cM_{\nu}^{ss}(\bar{k}_\nu)$. The $p$-divisible group $A^o[p^\infty]$ of $A^o$ then carries an $\Oo_F$-action $\iota^o[p^\infty]$ and a compatible polarization $\lambda^o[p^\infty]$ determined by $\iota^o$ and $\lambda^o$ respectively. Decompose $A^o[p^\infty]$ as in equation \eqref{eq:decompositionpdivisible} we get
\begin{equation}\label{eq:bXglobal}
    (\bX,\iota_\bX,\lambda_\bX):=(A^o[v_0^\infty],\iota^o[v_0^\infty],\lambda^o[v_0^\infty])
\end{equation}
where $\iota^o[v_0^\infty]$ is the $\Oo_{F,w_0}$-action determined by $\iota^o[p^\infty]$ and $\lambda^o[v_0^\infty]$ is the polarization of $A^o[v_0^\infty]$ determined by $\lambda^o[p^\infty]$. 

Let $W'$ be the $n$-dimensional Hermitian vector whose local invariants are the same as $W$ except at $v_0$ and $\varphi_0$ where it has signature $(0,n)$. Associate to $W'$ the group ${G'}^\Q$ as in \eqref{eq:G^Q} where we associate $G^\Q$ to $W$. Also define 
\begin{equation}\label{eq:V'}
    V':=\Hom_{F}(W_0,W')
\end{equation}
together with the naturally defined Hermitian form. Then define $\tilde{G}':=Z^\Q\times_{\mathbb{G}_m} {G'}^\Q$ which is an inner form of $\tilde{G}$. 
Let $\cN'$ be the Rapoport-Zink space of $p$-divisible groups with $\Oo_F$-actions and comptible principal polarizations satisfying the Kottwitz condition, the wedge condition and the Eisenstein condition with respect to the signature $((1,n-1)_{\{\varphi_0\}},(0,n)_{\Phi\backslash \{\varphi_0\}})$, defined by the framing object $(A^o[p^\infty],\iota^o[p^\infty],\lambda^o[p^\infty])$. 
Then we have the following uniformiation theorem.

\begin{theorem}\label{thm:uniformization}
We have 
\begin{equation}
    \cN'=Z^\Q(\Q_p)/K_{Z^\Q,p}\times (\cN \widehat{\times}_{\Spf\Oo_{\breve{F}_{w_0}}} \Spf \Oo_{\breve{E}_\nu})\times \prod_{v\neq v_0} \rU(V)(F_{0,v})/K_{G,v}
\end{equation}
where the product in the last factor is over all places of $F_0$ over $p$ not equal to $v_0$ and $\cN=\cN^{F_{w_0}/\Q_p}_{(1,n-1)}\cong \cN^{F_{w_0}/F_{0,v_0}}_{(1,n-1)}$. Here $\cN^{F_{w_0}/\Q_p}_{(1,n-1)}$ (resp. $\cN^{F_{w_0}/F_{0,v_0}}_{(1,n-1)}$) is defined in Definition \ref{def:RZspace} using the framing objects $(\bX,\iota_\bX,\lambda_\bX)$ in \eqref{eq:bXglobal} (resp. $\mathcal{C}_{\bar\bk_\nu}((\bX,\iota_\bX,\lambda_\bX))$). There is an isomorphism depending on the choice of base point $(A_0^o,\iota_0^o,\lambda_0^o,A^o,\iota^o,\lambda^o)\in \cM_{\nu}^{ss}(\bar{k}_\nu)$:
\begin{equation}\label{eq:uniformization}
    \Theta: \tilde{G}'(\Q)\backslash \cN'\times\tilde{G}(\bA_f^p)/K^p \cong \widehat{M}_\nu^{ss}.
\end{equation}
\end{theorem}
\begin{proof}
Using exactly the same proof of \cite[Lemma 8.16]{RSZdiagonal}, we know that for $v\neq v_0$ above $p$, we have
\[\cN^{F_v/\Q_p}_{(0,n)}\cong \rU(V)(F_{0,v})/K_{G,v}.\]
As a corollary we know that
\begin{equation}
    \cN'=Z^\Q(\Q_p)/K_{Z^\Q,p}\times (\cN^{F_{w_0}/\Q_p}_{(1,n-1)} \widehat{\times}_{\Spf\Oo_{\breve{F}_{w_0}}} \Spf \Oo_{\breve{E}_\nu})\times \prod_{v\neq v_0} \rU(V)(F_{0,v})/K_{G,v}.
\end{equation}
Then \eqref{eq:uniformization} is a special case of \cite[Theorem 6.30]{RZ}.
By Theorem \ref{thm:comparisonRZspaces}, we can also replace $\cN^{F_{w_0}/\Q_p}_{(1,n-1)}$ above by $\cN^{F_{w_0}/F_{0,v_0}}_{(1,n-1)}$.
\end{proof}

\begin{theorem}\label{thm:mainglobal}
Assume that $T\in \Herm_n(\Oo_F)_{>0}$ with $\mathrm{Diff}(T,V)=\{v_0\}$ where $v_0\in \mathscr{V}_{\mathrm{ram}}$ \eqref{eq:cVram}. Assume that $v_0$ is not over $2$ and is unramified over $\Q$. Then $\cZ(T)_{\mathrm{red}}$ is equidimensional of dimension $\frac{1}{2}\mathfrak{d}(L_{v_0})$ where $L_{v_0}$ is any Hermitian lattice over $\Oo_{F,v_0}$ whose gram matrix is $T$ and $\mathfrak{d}(L_{v_0})$ is defined as in Theorem \ref{thmA}. 
\end{theorem}
\begin{proof}
The proof resembles that of \cite[Proposition 11.2]{KR2}. By Proposition \ref{prop:precisesupport}, $\cZ(T)$ is supported on the basic locus over $\nu$ for those finite places $\nu$ of $E$ that induces $v_0$.
Fix such a $\nu$ and let $w_0$ be the place of $F$ above $v_0$. Choose a framing object $(A_0^o,\iota_0^o,\lambda_0^o,A^o,\iota^o,\lambda^o) \in \cM_{\nu}^{ss}(\bar{k}_\nu)$ which determines a supersingular formal $\Oo_{F,w_0}$-module $(\bX,\iota_\bX,\lambda_\bX)$ as in \eqref{eq:bXglobal}. 

Define $V'$ as in \eqref{eq:V'} and $G':=\rU(V')$. By Proposition \ref{prop:precisesupport}, we know that 
\[V'\cong V_T\] 
as Hermitian spaces. In particular, $V'_v\cong V_v$ as a Hermitian space for all finite places $v\neq v_0$ of $F_0$. 
We can thus think of 
\[L^{v_0}:=L\otimes_{\Oo_{F_0}} \hat{\Oo}_{F_0}^{v_0}\] 
(see \eqref{eq:globalL} for the definition of $L$) as a lattice in $V'(\bA^{v_0}_{F_0,f})$. Its stabilizer in $G'(\bA^{v_0}_{F_0,f})$ is $K^{v_0}_G$. On the other hand, by Proposition \ref{prop:comparisonofspecialcycles} we have the following identification.
\[V'_{v_0}\cong\Hom_{\Oo_{F,w_0}}(\bY,\bX)\otimes \Q\cong \Hom_{\Oo_{F,w_0}}(\mathcal{C}_{\bar{k}_\nu}(\bY),\mathcal{C}_{\bar{k}_\nu}(\bX))\otimes \Q =\bV.\]
Let $ \widehat{\cZ(T)}_\nu$ be the closure of $\cZ(T)\times_{\Spec \Oo_{E}}\Spec \Oo_{\breve{E}_\nu}$ in $\widehat{\cM}^{ss}_{\nu}$. Then by Theorem \ref{thm:uniformization} and the fact that $\tilde{G}'=Z^\Q\times \mathrm{Res}_{F_0/\Q} G'$, we have (c.f. \cite[Proposition 6.3]{KR2})
\[\widehat{\cZ(T)}_\nu\cong (Z^\Q(\Q)\backslash Z^\Q(\bA_f)/K_{Z^\Q}) \times \bigsqcup_{g\in G'(F_0)\backslash G'(\bA^{v_0}_{F_0,f})/K^{v_0}_G } \bigsqcup_{\substack{\bx\in \Omega(T)\\ g^{-1}\bx \in (L^{v_0})^n}} \cZ(\bx),\]
where $\cZ(\bx)$ is the special cycle of $\cN^{F_{w_0}/F_{0,v_0}}_{(1,n-1)}$ defined in Definition \ref{def:localspecialcycle} and 
\[\Omega(T):=\{\bx\in (V')^n\mid (\bx,\bx)=T\}.\]
Here we think of $V'$ as a subset of both $\bV$ and $V'(\bA^{v_0}_{F_0,f})$.
The theorem is now a consequence of Theorem \ref{thmA}.
\end{proof}

\bibliographystyle{alpha}
\bibliography{reference}

\end{document}